\documentclass[margincite]{sl2art}
\usepackage{marginfix}

\makeatletter
\def\qed@warning{}
\makeatother \usepackage{tikz-cd} \usepackage{todonotes}

\usepackage{graphicx}
\graphicspath{{fig/}}

\newcommand{\ZZ}{\mathbb{Z}}

\newcommand{\CC}{\mathbb{C}}

\renewcommand{\Re}{\operatorname{Re}}
\renewcommand{\Im}{\operatorname{Im}}

\newcommand{\iso}{\cong}

\newcommand{\set}[1]{\left\{ \def\given{\ \middle| \ }  #1 \right\}  }

\newcommand{\slg}{\mathrm{SL}_2(\CC)}
\newcommand{\pslg}{\mathrm{PSL}_2(\CC)}

\newcommand{\vol}{\operatorname{Vol}}
\newcommand{\cs}{\operatorname{CS}}
\newcommand{\cvol}{\operatorname{V}}

\newcommand{\volf}{\mathcal{V}}

\newcommand{\repv}{\mathsf{R}}
\newcommand{\decrepv}{\mathsf{D}}
\newcommand{\logrepv}{\mathsf{S}}

\newcommand{\mer}{\mathfrak{m}}
\newcommand{\lon}{\mathfrak{l}}
\newcommand{\lonb}{\tilde{\mathfrak{l}}}
\newcommand{\comp}[1]{S^3 \setminus #1}
\newcommand{\extr}[1]{E(#1)}

\NewDocumentCommand{\homol}{D[]{1} m D[]{} }{\operatorname{H}_#1(#2)^{#3}}
\NewDocumentCommand{\cohomol}{D[]{1} m D[]{} }{\operatorname{H}^#1(#2)^{#3}}

\newcommand{\dil}{\operatorname{Li}_2}
\newcommand{\dilr}{\mathcal{L}}

\NewDocumentCommand{\qsl}{O{\xi}}{\mathcal{U}_{#1}(\mathfrak{sl}_2)}
\NewDocumentCommand{\weyl}{O{\xi}}{\mathcal{W}_{#1}}

\newcommand{\defeq}{:=}
\DeclareMathOperator{\tr}{tr}

\newcommand{\unknot}{\bigcirc}
\newcommand{\lens}[1]{\Lambda(#1)}
\newcommand{\storus}{W}

\newcommand{\cf}[1]{\left[ #1 \right]}

\declaretheorem[style=bigtheorem,title=Theorem,refname={Theorem,Theorems},Refname={Theorem,Theorems}]{bigtheorem}
\declaretheorem[style=theorem]{proposition}
\declaretheorem[style=theorem,sibling=proposition]{theorem}
\declaretheorem[style=theorem,sibling=proposition]{lemma}

\declaretheorem[style=theorem,sibling=proposition]{corollary}
\declaretheorem[style=definition,sibling=proposition]{definition}
\declaretheorem[style=definition,sibling=proposition]{remark}

\usepackage{biblatex}
\title{Surgery calculus for classical \(\operatorname{SL}_2(\mathbb{C})\) Chern-Simons theory}
\author{Calvin McPhail-Snyder}{Duke University\\UNC Chapel Hill}
\contact{calvin@esselltwo.com}{www.esselltwo.com}

\begin{document}

\abstract{
  Classical \(\slg\)-Chern-Simons theory assigns a \(3\)-manifold \(M\) with representation \(\rho : \pi_1(M) \to \slg\) its \defemph{complex volume} \(\cvol(M, \rho) \in \CC / 2 \pi^2 i \ZZ\), with real part the volume and imaginary part the Chern-Simons invariant.
  The existing literature focuses on computing \(\cvol\) using a triangulation.
  In this paper we show how to compute \(\cvol(M, L, \rho)\) directly from a surgery diagram for \(M\) a compact oriented \(3\)-manifold with torus boundary components, embedded cusps \(L\), and representation \(\rho : \pi_1(M \setminus L) \to \slg\).
  When \(M\) has nonempty boundary \(\cvol(M, L, \rho)(\mathfrak{s})\) depends on some extra data \(\mathfrak{s}\) we call a \defemph{log-decoration}.
  Our method describes \(\rho\) in a coordinate system closely related to quantum groups, and we think of our construction as a classical, noncompact version of Witten-Reshetikhin-Turaev's quantum \(\operatorname{SU}(2)\) Chern-Simons theory.
}

\maketitle

\section{Introduction}
\subsection{The complex volume}
Let \(M\) be a compact, oriented, hyperbolic \(3\)-manifold of finite volume; being hyperbolic means that \(M\) is equipped with a complete Riemannian metric of curvature \(-1\).
The \defemph{hyperbolic volume} of \(M\) is the volume determined by the metric.
By Mostow-Prasad rigidity the hyperbolic metric is actually a topological invariant of \(M\), so the volume is as well.

There is a natural generalization of the volume which comes from a more algebraic description of the hyperbolic structure.
Because the isometry group of hyperbolic  \(3\)-space is \(\pslg\) the hyperbolic structure is equivalent to the choice of a representation \(\rho : \pi_1(M) \to \pslg\) up to conjugacy, which is again uniquely determined by \(M\).
We call it the \defemph{holonomy} of the hyperbolic structure.
When \(M\) is orientable it is spin so this representation always lifts to \(\slg\), and choices of lifts are naturally in bijection with spin structures on  \(M\).

Now let \(A\) be a flat \(\mathfrak{sl}_2\)-connection with holonomy \(\rho\).
The \defemph{complex volume}
\begin{equation}
  \label{eq:chern-simons-integral}
  \begin{aligned}
    \cvol(M, \rho)
    &= 
    \vol(M, \rho) + i \cs(M, \rho)
    \\
    &=
    -\frac{i}{4} \int_{M} \tr [A, dA] + \frac{2}{3} \tr [A, A \wedge A] \in \CC/i\pi^2 \ZZ
  \end{aligned}
\end{equation}
is an invariant of \((M, \rho)\).
The complex volume is also called the \defemph{complex Chern-Simons invariant}; its imaginary part is the usual Chern-Simons invariant.

The integral \eqref{eq:chern-simons-integral} makes sense for any \(\rho\), not just the one coming from the hyperbolic structure on \(M\).
It can be extended \cite{Meyerhoff1986} to the case where \(M\) has cusps, which means that \(M\) is noncompact but \(M = \hat M \setminus L\) for \(\hat M\) a compact orientable manifold  and \(L\) a link in \(\hat M\).
There are some subtleties with normalization: depending on whether \(M\) has cusps, \(\cvol(M, \rho)\) might be defined modulo \(\pi^2 i\) or \(2\pi^2 i\).
However, a choice of lift to \(\slg\) resolves this:
\begin{bigtheorem}
  \label{thm:cvol-simple}
  Let \(M\) be a closed oriented \(3\)-manifold, \(L\) a link in \(M\) (which we think of as a set of cusps) and \(\rho : \pi_1(M \setminus L) \to \slg\) a representation.
  We require that \(\rho\) is \defemph{parabolic along \(L\)}, which means that \(\tr \rho(\mer) = \pm 2\) for any meridian \(\mer\) of any component of \(L\).
  We allow the case that \(\rho(\mer) = \pm 1\) is trivial.

  There is a version of the complex volume
  \begin{equation}
    \cvol(M, L, \rho) \in \CC/2\pi^2 i \ZZ
  \end{equation}
  well-defined modulo \(2\pi^2i\).
  It depends only on the conjugacy class of \(\rho\), and if \(\overline{M}\) is \(M\) with the opposite orientation, then
  \begin{equation*}
    \cvol(\overline{M}, L, \rho)
    =
    \overline{\cvol(M, L, \rho)}.
    \qedhere
  \end{equation*}
\end{bigtheorem}
We do not require \(M\) to be hyperbolic here, although of course if we evaluate \(\cvol(M, L, \rho)\) for \(\rho\) a lift of the complete hyperbolic structure of the cusped hyperbolic manifold \((M,L)\) we obtain the usual complex volume.
Forgetting the choice of lift results in a version of \(\cvol\) with values instead in \(\CC/\pi^2 i \ZZ\).

\Cref{thm:cvol-simple} is already known in various forms; we give some references later in the introduction.
The aims of this paper are to
\begin{enumerate}
  \item explain how to compute \(\cvol\) from surgery presentations, not triangulations,
  \item show that the definition makes sense for geometrically degenerate \(\rho\),
  \item discuss connections to quantum \(\mathfrak{sl}_2\) and surgery TQFT, and
  \item extend the definition to manifolds with torus boundary: in this case \(\cvol\) also depends on some extra structure called a \defemph{log-decoration}.
\end{enumerate}
Many of these results were known (perhaps implicitly) to experts: our goals are to study them systematically and to show how they relate to the shape coordinate formalism of \cite{McPhailSnyder2022}.
Some of our proofs (in particular, of invariance under the Reidemeister III and blow-up moves) still rely on triangulations, so our work is not (yet) independent of the triangulation formalism.

\subsection{Manifolds with boundary and log-decorations}
To compute \(\cvol\) it is useful to be able to cut and glue \(M\) along embedded tori, and 
to do this we need to extend  \(\cvol\) to the case where \(M\) has torus boundary components.
When we do \(\cvol\) depends on some extra boundary data.

\begin{definition}
  Let \(M\) be a compact oriented \(3\)-manifold and \(L\) a link in the interior of \(M\).
  Assume that \(\partial M = T_1 \amalg \cdots \amalg T_n\) is a disjoint union of tori.
  Each \(\pi_1(T_j) \iso \ZZ^2\) is abelian, so \(\rho(\pi_1(T_j))\) is always conjugate to the subgroup \(B \subset \slg\) of upper-triangular matrices.
  A \defemph{decoration} \(\delta\) of \(\rho\) is an identification of \(\rho(\pi_1(T_j))\) with a subgroup of \(B\).
\end{definition}
The original definition \cite[Section 4]{Garoufalidis2015} of decoration is slightly different;
here we are using the characterization of \cite[Proposition 4.6]{Garoufalidis2015}.
Observe that a decoration \(\delta\) gives a choice \(\delta(x)\) of eigenvalue of \(\rho(x)\) for every \(x \in \pi_1(T_j)\), since%
\note{
  Our choice to use lower-triangular matrices here is nonstandard but matches the conventions of \cite{McPhailSnyder2022}.
  It does not affect the theory in any significant way.
}
\[
  \rho(x) \text{ is conjugate to }
  \begin{pmatrix}
    \delta(x) & 0 \\
    * & \delta(x)^{-1}
  \end{pmatrix}
\]
Another way to phrase this is to say that \(\rho\) induces a homomorphism
\[
  \rho : \homol{\partial M;\ZZ} = \bigoplus_{j=1}^{n} \homol{T_j;\ZZ} \to \slg.
\]
and a decoration induces a lift
\[
  \delta : \homol{\partial M;\ZZ} \to \CC^{\times}
\]
which by Poincar\'{e} duality we can think of as a cohomology class \(\delta \in \cohomol{\partial M;\CC^{\times}}\).
\begin{definition}
  A \defemph{log-decoration} \(\mathfrak{s}\) of \(\rho\) is a cohomology class \(\mathfrak{s} \in \cohomol{\partial M; \CC}\) with
  \begin{equation}
    \label{eq:boundary-log-compatibility}
    \exp(2 \pi i \mathfrak{s}(x)) = \delta(x)
    \text{ for any } x \in \homol{\partial M; \ZZ}
  \end{equation}
  for some decoration \(\delta\) of \(\rho\).
\end{definition}

In practice, choosing a meridian \(\mer\) and longitude \(\lon\) gives a basis of each \(\homol{T;\ZZ}\), so \(\mathfrak{s}\) is determined by the values \(\mathfrak{s}(\mer) = \mu\) and \(\mathfrak{s}(\lon) = \lambda\).

\begin{remark}
  The group \(\cohomol{\partial M; \ZZ}\) acts on the space of log-decorations for \(\rho\) via
  \[
    (\alpha + \mathfrak{s})(x) = \mathfrak{s}(x) + \alpha(x)
    \text{ for }
    \alpha \in \cohomol{\partial\extr L; \ZZ}
  \]
  so we can think of them as generalized spin structures.
  They are closely related to the cohomology classes appearing in non-semi-simple TQFT \cite{Blanchet2016}.
\end{remark}

\begin{remark}
  Here is a slightly different perspective.
  Let \(\repv(M, L)\) be the representation variety of \(M \setminus L\), that is the space of homomorphisms \(\rho : \pi_1(M \setminus L) \to \slg\).
  The space of decorated representations \(\decrepv(M, L)\) covers \(\repv(M,L)\); this cover is generically 2-1, but is more complicated at certain points, such as boundary-parabolic \(\rho\) where the eigenvalues are \(\pm 1\).
  However, if we use \(\decrepv(M,L)\) as the base the space \(\logrepv(M, L)\) of log-decorated representations is a genuine covering space
  \[
    \cohomol{\partial M; \ZZ} \to \logrepv(M, L) \to \decrepv(M, L)
  \]
  with deck transformation group \(\cohomol{\partial M; \ZZ}\).
  When \(L = \emptyset\), \(\logrepv(M, \emptyset)\) is the space \(\mathcal{R}(M)\) which \citeauthor{Marche2012} \cite{Marche2012} uses to give a geometric definition of the complex volume.
\end{remark}

We can now state our generalization of \cref{thm:cvol-simple}.

\begin{bigtheorem}
  \label{thm:cvol-for-manifolds-with-boundary}
  Let \(M\) be a manifold with torus boundary components \(\partial M = T_1 \amalg \cdots \amalg T_n\).
  Choose \(L\) a link in \(M\) and \(\rho : \pi_1(M \setminus L) \to \slg\) a representation parabolic along \(L\).
  Then:
  \begin{enumerate}
    \item the complex volume \(\cvol(M, L, \rho)(\mathfrak{s}) \in \CC/2\pi^2 i \ZZ\)  is well-defined once we choose a log-decoration \(\mathfrak{s}\) for \(\rho\).
      It depends on \(\rho\) only up to conjugacy.
    \item If \(\overline{M}\) is \(M\) with the opposite orientation, then
      \begin{equation}
        \cvol(\overline{M}, L, \rho)
        =
        -\cvol(M, L, \rho).
      \end{equation}

    \item
  If \(\mathfrak{s}\) and \(\mathfrak{s}'\) are two different log-decorations, then
      \begin{equation}
        \label{eq:boundary-log-dependence}
        \cvol(M, L, \rho)(\mathfrak{s}')
        -
        \cvol(M, L, \rho)(\mathfrak{s})
        \equiv
        4\pi^2 i\sum_{j=1}^{n}
        \Delta \lambda_j \mu_j - \Delta\mu_j \lambda_j
        \pmod{2 \pi^2 i \ZZ}
      \end{equation}
      where \(\mu_j = \mathfrak{s}(\mer_j)\), \(\lambda_j = \mathfrak{s}(\lon_j)\) and similarly for \(\mathfrak{s}'\),
      \((\mer_j, \lon_j)\) is an oriented meridian-longitude%
      \note{
        Meridians \(\mer_j\) make sense for any \(M\): they are generators of the kernel of the inclusion \(\partial M \to M\) on homology.
        We then choose \(\lon_j\)  to be an independent element of \(\homol{T_j;\ZZ}\) so that \(\mer_j, \lon_j\) is appropriately oriented.
      }
      pair for the \(j\)th component of \(\partial M\),
      and \(\Delta \lambda_j = \lambda_j ' - \lambda_j\), \(\Delta \mu_j = \mu_j' - \mu_j\).
    \item Let \(L'\) be the link obtained from \(L\) by removing all the components along which \(\rho\) is trivial.
      Then \(\cvol(M, L, \rho) = \cvol(M, L', \rho) \) for every log-decoration.\qedhere
  \end{enumerate}
\end{bigtheorem}

\subsection{Cutting and gluing}

Now suppose that \(M\) is obtained by gluing two other manifolds \((M_1, L_1, \rho_1)\) and \((M_2, L_2, \rho_2)\) along some boundary components.
By this we mean that 
\[
  M = M_1 \cup_h M_2
\]
where \(X_k \subset \partial M_k\) are submanifolds of the boundaries and the identification is made along an orientation-reversing homeomorphism \(h : X_1 \to X_2\).
This identification assembles a new link \(L\) out of \(L_1\) and \(L_2\) and builds a representation \(\rho : \pi_1(M \setminus L) \to \slg\) out of \(\rho_1\) and \(\rho_2\).

\begin{bigtheorem}
  \label{thm:gluing}
  Let \((M, L, \rho)\) be the manifold obtained by gluing \((M_1, L_1, \rho_1)\) and \((M_2, L_2, \rho_2)\) as above.
  Then if \(\mathfrak{s}_1, \mathfrak{s}_2\) are log-decorations for \(M_1, M_2\) compatible (defined below) with the identification,
  \[
    \cvol(M, L, \rho)(\mathfrak{s})
    =
    \cvol(M_1, L_1, \rho_1)(\mathfrak{s}_1)
    +
    \cvol(M_2, L_2, \rho_2)(\mathfrak{s}_2)
  \]
  where \(\mathfrak{s}\) is determined by \(\mathfrak{s}_1\) and \(\mathfrak{s}_2\) in the obvious way.%
  \note{
    To make it obvious: \(\mathfrak{s}\) is an element of \(\cohomol{\partial M; \CC}\) and \(\partial M\) consists of the remaining boundary components of \(\partial M_1 \amalg \partial M_2\) after the identification, so  \(\mathfrak{s}\) is just the image of \(\mathfrak{s}_1 \oplus \mathfrak{s}_2 \in \cohomol{\partial M_1; \CC} \oplus \cohomol{\partial M_2; \CC}\) in \(\cohomol{\partial M; \CC}\).
  }
\end{bigtheorem}
\begin{definition}
  Let \(T\) be a component of \(X_1\) identified with a component \(T'\) of \(X_2\) when gluing \(M_1\) and \(M_2\).
  Choose meridian-longitude pairs \((\mer, \lon)\) and \((\mer', \lon')\) for them.
  During the gluing we identify \(p \mer + q \lon\) with \(\lon'\) and \(r \mer + s \lon\) with \(\mer'\) for integers \(p,q,r,s\in \ZZ\) with \(ps - qr = 1\).
  We say that log-decorations \(\mathfrak{s}_k\) for the \(M_k\) are \defemph{compatible} with the gluing if
  \begin{equation}
    \label{eq:boundary-log-dehn-filling}
    \mathfrak{s}_1(p \mer + q \lon) = \mathfrak{s}_2(\lon ')
    \text{ and }
    \mathfrak{s}_1(r \mer + s \lon) = \mathfrak{s}_2(\mer ')
  \end{equation}
  for each identified component \(T\).
\end{definition}
This condition is quite natural: for our representations \(\rho_1, \rho_2\) to give a well-defined representation of the glued manifold we must have
\[
  \rho_1 (\mer )^{p} \rho_1(\lon)^{q} = \rho_2(\lon')
  \text{ and }
  \rho_1 (\mer )^{r} \rho_1(\lon)^{s} = \rho_2(\mer')
\]
and \cref{eq:boundary-log-dehn-filling} is simply the logarithm of this equation.

Dehn filling is the special case where \(M_2\) is a disjoint union of solid tori, which suggests the method we use to define \(\cvol\) and compute it in practice.
We first define an invariant \(\volf(L_0, \rho_0)(\mathfrak{s})\), where \(L_0\) is a link in \(S^3\), \(\rho_0 : \pi_1(S^3 \setminus L_0) \to \slg\) is a representation, and \(\mathfrak{s}\) is a log-decoration of the link exterior \(\extr{L_0}\).
(\(\extr{L_0}\) is the complement of an open regular neighborhood of \(L_0\), so it is a compact submanifold of \(S^3\) with boundary a union of tori.)
We think of \(\volf\) as an invariant of link exteriors, and as a special case the value \(\volf(\unknot, \rho_0)(\mathfrak{s})\) on the unknot exterior is the complex volume of a solid torus.
Because we can obtain any \((M, L, \rho)\) by gluing solid tori to the exterior of some link \(L_0\), \cref{thm:gluing} determines \(\cvol(M, L, \rho)\) in terms of \(\volf(L_0, \rho_0)(\mathfrak{s})\) and of the added solid tori.\note{The complex volume of a solid torus is closely related to the length correction terms \(\sum_j \lambda_j\) in \cite[Theorem 14.5]{Neumann2004}.}

This construction corresponds to a representation of \((M, L, \rho)\) as a decorated link \(L_0\) in \(S^3\), where components are labeled by one of 
\begin{itemize}
  \item rational numbers \(p/q\), indicating Dehn surgery along them,
  \item \(1/0 = \infty\), indicating components of \(L\), i.e.\@ cusps of \(M\), and
  \item nothing, indicating torus boundary components.
\end{itemize}
The usual surgery calculus \cite{Gompf1999} naturally extends to these links.
We can similarly describe \(\rho\) in terms of a representation \(\rho_0 : \pi_1(S^3 \setminus L_0) \to \slg\), which can be given by a decoration of a diagram of \(L_0\) by complex numbers called a \defemph{shaping}.
To motivate these we first need to discuss the construction of \(\volf\) in more detail, which is done in the next subsection.

First we state another gluing result.
Taking the connected sum of two pairs \((M_1, L_1)\) and \((M_2, L_2)\) gives a manifold  \(M = M_1 \# M_2\) with a new link \(L = L_1 \amalg L_2\) inside it.
Similarly, representations \(\rho_k : \pi_1(M_k \setminus L_k) \to \slg\) give a new representation \(\rho = \rho_1 * \rho_2\) on the connect sum, and because \(\partial M = \partial M_1 \amalg \partial M_2\) we can combine log-decorations \(\mathfrak{s}_k\) for each \(\rho_k\) to one \(\mathfrak{s} = \mathfrak{s}_1 \oplus \mathfrak{s}_2\) for \(\rho\).
\begin{bigtheorem}
  \label{thm:connect-sum}
  The complex volume is additive under connect sum:
  \begin{equation*}
    \cvol(M_1 \# M_2, L_1 \amalg L_2, \rho_1 * \rho_2)(\mathfrak{s}_1 \oplus \mathfrak{s}_2)
    =
    \cvol(M_1, L_1, \rho_1)(\mathfrak{s}_1)
    +
    \cvol(M_2, L_2, \rho_2)(\mathfrak{s}_2).
    \qedhere
  \end{equation*}
\end{bigtheorem}

\begin{remark}
  Theorems \ref{thm:gluing} and \ref{thm:connect-sum} are similar to the cutting-and-gluing properties of a TQFT (topological quantum field theory),%
  \note{
    Usually TQFTs are multiplicative under gluing and connect sum, not additive, and they usually take values in \(\CC\), not \(\CC/2\pi^2 i \ZZ\).
    To fix both of these discrepancies consider \(\exp(\cvol/\pi)\) instead of \(\cvol\).
  }
  and our surgery presentations of \((M, L)\) closely resemble those for a TQFT constructed from a modular category \cite{Bakalov2001}, sometimes called a \defemph{surgery} or \defemph{Reshetikhin-Turaev} TQFT.

  However, there are some differences.
  The most obvious is the dependence on \(\rho\): the theory is no longer topological, but geometric, because \(\rho\) is a choice of generalized hyperbolic structure.
  \citeauthor{Turaev2010} \cite{Turaev2010} has previously considered such theories in the framework of \defemph{homotopy quantum field theory}.

  Another difference is that we assign a manifold with nonempty boundary a function \(\cvol(M, L, \rho)\) on the space of log-decorations, which are a torsor over \(\cohomol{\partial M; \ZZ}\).
  In a quantum field theory we would have instead have a linear map between tensor products of vector spaces of states, with each space corresponding to a boundary component.
  \citeauthor{Freed1995} \cite{Freed1995} showed that classical Chern-Simons theory assigns complex lines to the boundary components.
  It would be quite interesting to clarify the relationship between log-decorations and these lines.

    The motivating example for surgery TQFTs is Witten--Reshetikhin--Turaev theory \cite{Witten1989,Reshetikhin1991}.
    Both WRT and complex volume are related to Chern-Simons theory: one motivation for our work is to clarify these connections, which are related to the volume conjecture \cite{Murakami2010}.
  One important difference between \(\cvol\) and WRT theory is that we use the noncompact complex group \(\slg\), not its compact real form \(\operatorname{SU}(2)\).
  In addition, as mentioned above we should think of \(\cvol\) as a classical, not quantum field theory.
  The corresponding quantum field theory has not yet been constructed, but should generalize the link invariants of \citeauthor{Blanchet2018} \cite{Blanchet2018}.
  The link invariants can be computed in terms of \defemph{hyperbolic tensor networks} \cite{McPhailSnyder2022b} closely related to our computation of \(\volf\) and \(\cvol\).
\end{remark}

\subsection{Computing the volume}

Computing \(\cvol\) directly from the integral \eqref{eq:chern-simons-integral} is quite difficult.
Instead we triangulate \(M\) and compute the integral on each piece.
The representation \(\rho\) is encoded by assigning the tetrahedra \defemph{shape parameters} \(z \in \CC\) which describe how to embed them in hyperbolic space; when \(z = 0, 1\) the tetrahedron is geometrically degenerate.
The volume \(\vol(M, \rho) = \Re \cvol(M, \rho)\) can then be computed by a sum \(\sum_j D(z_j)\) over tetrahedra, where \(D\) is a version of the dilogarithm \cite{Zagier2007}, specifically the imaginary part%
\note{%
  There is an unfortunate factor of \(i\), so the \emph{real} part of \(\cvol\) comes from the \emph{imaginary} part of \(\dilr\).
}
of a function \(\dilr\) on \(\CC \setminus \set{0,1}\).

Computing the imaginary part \(\cs(M, \rho) = \Im \cvol(M, \rho)\) is significantly more difficult: we have to pass from \(D = \Im \dilr \) to the full dilogarithm \(\dilr\).
The function \(\dilr\) is holomorphic on a somewhat complicated covering space of \(\CC \setminus \set{0,1}\).
To deal with this we need to pick extra combinatorial data on the triangulation called a \defemph{flattening}, which is roughly a coherent choice of logarithms of the shape parameters \(z_j\).
Once this is done we can compute 
\begin{equation}
  \label{eq:dilog-sum-simple}
  \cvol(M, \rho) = -i \sum_j \dilr(\zeta^0_j, \zeta^1_j)
\end{equation}
in terms of the flattening \((\zeta_j^0, \zeta_j^1)_j\).
This approach is due to \citeauthor{Neumann2004} \cite{Neumann2004} and can be stated in direct geometric terms \cite{Marche2012}.

This method allows us to compute \(\cvol(M, L, \rho)\) for \(M\) a cusped manifold; strictly speaking we have been using \defemph{ideal} triangulations whose vertices lie on the cusps \(L\).
However, if the shapes \(z_j\) are deformed to instead give a representation on some Dehn filling of \((M,L)\) the sum \eqref{eq:dilog-sum-simple} will compute the volume of the filled manifold.
There are some subtleties about the flattening, which are related to our condition \eqref{eq:boundary-log-compatibility} on the log-decorations.
(See also the final condition of \cite[Theorem 14.7]{Neumann2004}.)

So far our discussion has focused on tetrahedra and triangulations of \(M\).
These are convenient for computer use: this computation has been implemented in {\scshape SnapPy} \cite{SnapPy}.
However, triangulations are somewhat hard for humans to use compared to surgery presentations.
One of our goals in writing this paper is to explain how to use \citeauthor{Neumann2004}'s simplicial formula directly from a link diagram.

Considerable progress in this direction has been made using the optimistic limit method of \citeauthor{Yokota2011} \cite{Yokota2011}.
Most recently \citeauthor{Cho2013} \cite{Cho2013} showed how to compute the complex volume of a hyperbolic link \(L\) in \(S^3\) directly from a diagram, and \citeauthor{Yoon2018} \cite{Yoon2018} extended this method to Dehn fillings of \(L\).
The key ingredient is the \defemph{octahedral decomposition} \cite{ThurstonDNotes, Kim2016}, which assigns an ideal triangulation to any diagram \(D\) of \(L\).
The shape parameters of the tetrahedra are expressed in terms of complex variables associated to parts of the diagram.

We use a version of these variables related to quantum groups.
Previous work of the author \cite{McPhailSnyder2022} explains how to describe decorated representations \(\rho : \pi_1(S^3 \setminus L) \to \slg\) by labeling the segments (edges) of a diagram of \(L\) by \defemph{shapes}, which are triples \(\chi = (a, b, m)\) of nonzero complex numbers.
The shapes of a diagram are required to satisfy some algebraic relations at each crossing which guarantee that the gluing equations (in the sense of \citeauthor{Neumann1985} \cite{Neumann1985}) of the octahedral decomposition are satisfied.

One advantage of doing this is that we can easily find flattenings by taking the logarithms of the shape parameters.
This works because the shape coordinates are \defemph{deformed Ptolemy coordinates}%
  \note{Because we do not use their results directly, we do not give a formal definition of Ptolemy coordinates, but they were the method used to determine the formula of Theorem \ref{thm:our-flattening}.}
  in the sense of \citeauthor{Zickert2007} \cite{Zickert2007, Zickert2016} and \citeauthor{Yoon2018} \cite{Yoon2018}.
  Without using Ptolemy coordinates it is not clear that flattenings even exist: Neumann originally showed \cite[Section 9]{Neumann2004} they do using a somewhat complicated combinatorial argument.

Now given a link \(L\) in \(S^3\) and a diagram \(D\) of \(L\) we can encode \(\rho : \pi_1(S^3 \setminus L) \to \slg\) by assigning shapes to \(D\).
Using these we obtain a flattening \(\mathfrak{f}\), and we can define the complex volume of \(\extr L\) as a sum over the tetrahedra
\begin{equation}
  \label{eq:volf-def-informal}
  \volf(L, \rho)(\mathfrak{s}) = -i \sum_{j} \dilr(\zeta_j^0, \zeta_j^1)
\end{equation}
where \(\mathfrak{s}\) is the log-decoration determined by \(\mathfrak{f}\) (as in \cref{def:induced-boundary-flattening}).
We can show that \(\volf\) depends only on \(\mathfrak{s}\), which is considerably simpler than the full flattening \(\mathfrak{f}\).
By Dehn filling some components of \(L\), viewing them as cusps, or leaving them unfilled this definition extends to the general case \(\cvol(M', L', \rho')\).

When we formally define \(\volf\) we will organize the sum in terms of crossings of the diagram, not tetrahedra.
This is more natural when working with diagrams (which can be broken down into crossings but no further) but has other advantages.
For example, when using earlier methods \cite{Zickert2016, Yoon2018} to compute \(\cvol\) we need to avoid certain geometrically degenerate shapings.
However, because our triangulations and flattenings are particularly regular we can make sense of these limits.
This is useful for technical reasons, and it also lets us define \(\cvol(M, L, \rho)\) for \emph{all} representations \(\rho\), not just geometrically nondegenerate ones.

\subsection{Plan of the paper}
\begin{description}
  \item[Section 2] We give an overview of ideal triangulations and their flattenings, then briefly explain how to use shape coordinates and flattenings to obtain flattened ideal triangulations of link complements.
  \item[Section 3] We show how to compute the complex volume \(\volf(L, \rho)(\mathfrak{s})\)  of a link exterior and study the dependence on the log-decoration \(\mathfrak{s}\).
  \item[Section 4] To motivate our definition for general \(3\)-manifolds we compute the complex volumes of solid tori and of lens spaces.
  \item[Section 5] We define the complex volume \(\cvol(M, L, \rho)\) of a compact, oriented \(3\)-manifold, possibly with torus boundary components and/or an embedded link \(L\), then give the proofs of our main theorems.
\end{description}

\subsection*{Acknowledgements}
Thanks to:
\begin{itemize}
  \item Christian Zickert and Matthias Goerner for some helpful discussions about complex volumes and flattening, and for sharing some unpublished results on the dependence of Neumann's volume formula on the flattening related to \cref{eq:boundary-log-dependence}.
  \item Adam S.\@ Levine, Eylem Yıldız, and David Rose for advice about surgery calculus.
\end{itemize}

\section{Shaped tangles and flattenings}
Our first step is define the complex volume \(\volf( L, \rho)\) of the exterior \(\extr L\) of a link \(L\) in \(S^3\).
To do this we need use a slightly nonstandard description of \(\rho\), which comes from decorating a diagram of \(L\) with \defemph{shapes} satisfying certain relations at the crossings.
These coordinates are closely related to a certain ideal triangulation of \(\comp L\) determined by \(D\) called the \defemph{octahedral decomposition}.
In this section we briefly discuss the relevant parts of this story; we refer to previous work \cite{McPhailSnyder2022} for more details, motivation, and a connection to quantum groups.
We first give some background on ideal triangulations.

\subsection{Ideal triangulations}

\begin{definition}
  Let \(L\) be a link in \(S^3\).
  An \defemph{ideal triangulation} is a triangulation \(\mathcal{T}\) of \(\comp L\) by ideal tetrahedra (\(3\)-simplices with their vertices removed) so that the missing vertices lie on \(L\).%
  \note{
    More formally, consider the space \(\widetilde{\comp L}\) obtained by collapsing each component of \(L\) to a point.
    An ideal triangulation is a triangulation of \(\widetilde{\comp L}\) in which all the vertices lie on the image of \(L\).
    We can also consider ideal triangulations of spaces other than link complements; technically speaking the octahedral decomposition is an ideal triangulation of \(\comp L\) minus two points.
  }
  An ideal tetrahedron \(\tau \in \mathcal{T}\) is \defemph{ordered} if we pick a labeling of its vertices by \(\set{0,1,2,3}\).
  This induces a total ordering on the vertices, hence an orientation of \(\tau\).
  If this orientation agrees with the orientation of \(\tau\) coming from \(\comp L\) then we assign \(\tau\) the sign \(+1\), and if they differ we assign \(\tau\) the sign \(-1\).
  From now on an ideal triangulation \(\mathcal{T} = \set{\tau_j}_j\) includes a choice of orderings of each tetrahedron and in particular the induced signs \(\epsilon_j\).
\end{definition}

\begin{marginfigure}
  \centering
  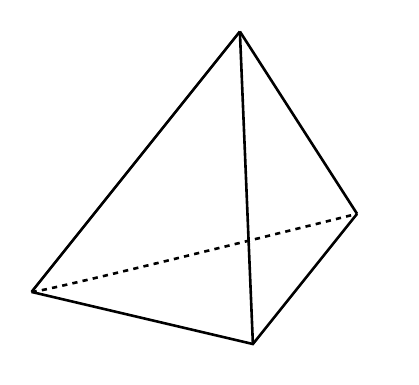
  \caption{Shape parameters assigned to the edges of a labeled tetrahedron.}
  \label{fig:tetrahedron-conventions-small}
\end{marginfigure}

\begin{definition}
  A \defemph{shaped} tetrahedron is an ordered tetrahedron \(\tau\) along with a choice of \defemph{shape} \(z^0 \in \CC \cup \set \infty\).
  We assign  \(z^0\) to the edges \(01\) and  \(23\).
  Edges \(12\) and \(03\) are assigned \(z^1\) and edges \(02\) and \(13\) are assigned \(z^2\), where
  \begin{equation}
    (z^{1})^{\epsilon} = \frac{1}{1- (z^0)^{\epsilon}}
    \text{ and }
    (z^{2})^{\epsilon} = 1- \frac{1}{(z^0)^{\epsilon}}
  \end{equation}
  and \(\epsilon\) is the sign of \(\tau\).
  These conventions are summarized in Figure \ref{fig:tetrahedron-conventions-small}.
  We say \(\tau\) is \defemph{degenerate} if one (hence all of) its shape parameters is \(0\), \(1\), or \(\infty\).

  An ideal triangulation \(\mathcal{T}\) is \defemph{shaped} if its tetrahedra are all assigned shapes so that at each edge \(e\) of \(\mathcal{T}\),
  \begin{equation}
    \label{eq:edge-gluing-equation}
    \prod_{j} z_j^{\nu(j)} = 1
  \end{equation}
  where the product is over all tetrahedra \(\tau_j\) glued to edge \(e\) and \(\nu(j)\) is \(0\), \(1\), or \(2\) depending on the type of edge glued from \(\tau_j\).
\end{definition}

The system of equations \eqref{eq:edge-gluing-equation} for all the edges of \(\mathcal{T}\) are sometimes called the Neumann--Zagier \cite{Neumann1985} gluing equations.
The shape of a tetrahedron describes its hyperbolic structure and the gluing equations guarantee that the structures on the tetrahedra assemble together to give one on \(\comp L\).

\begin{remark}
  We have picked a somewhat nonstandard convention on signs for our tetrahedra, but it turns out be convenient.
  It is known \cite[Proposition 5.7]{Zickert2007} that when considering the fundamental classes of a triangulated \(G\)-manifold%
  \note{%
    By \(G\)-manifold we mean a manifold \(M\) with a representation \(\pi_1(M) \to G\).
  }
  \(M\) one must consider simplices with vertex orderings.
  Each simplex \(\Delta\) inherits two orientations: one from the vertex ordering and one from \(M\).
  If the match, we say \(\Delta\) is positive, and if they don't it is negative.

  When \(G = \pslg\), we associate a shape parameter to \(\Delta\) by taking the cross-ratio of its vertices.
  When \(\Delta\) is negative, we need to take the cross-ratio in a different order to match the orientation of \(M\), and an edge of \(\Delta\) with shape \(z\) contributes \(z^{-1}\) to the gluing equation instead of \(z\).
  It turns out to be more convenient to take a different convention: we define the shape of a negative tetrahedron to be the \emph{inverse} of the cross-ratio of the usual vertex order.
  One advantage is that we can avoid signs in our gluing equations.
  Another is that then we get the usual relationships
  \begin{align}
    z^{1} &= \frac{1}{1-z^0}
          & 
    z^{2} &= 1 - \frac{1}{z^0}
          &
    \text{for } \epsilon &= 1
    \\
  \intertext{with a slight variation for negative tetrahedra:}
    z^{1} &= 1-\frac{1}{z^0}
          & 
    z^{2} &= \frac{1}{1 - z^0}
          &
    \text{for } \epsilon &= -1.
  \end{align}

  However, we still need to remember that a shaped tetrahedron \((\Delta, z, -1)\) is secretly one with shape \(1/z\), so we compute the volume of \((\Delta, z, \epsilon)\) as  \(\epsilon \Im R(z^{\epsilon})\), where \(R\) is the Rogers dilogarithm.
  It turns out that \(-\Im R(1/z) = \Im R(z)\), so this complication is not so important when computing ordinary hyperbolic volume.
  However, when computing \emph{complex} volume the identity \(R(1/z) = -R(z)\) no longer holds in general; we need to evaluate our lifted dilogarithm on the value \(1/z\).
  This is why we will later (in \cref{def:tetrahedron-flattening}) define a flattening to be a logarithm of \(z^{\epsilon}\), not of \(z\).
\end{remark}

\subsection{Flattenings and complex volume}

In order to compute the imaginary part of the complex volume we need an extra choice of structure on our shaped ideal triangulation called a flattening.
Flattenings are related to but more complicated than log-decorations.

\begin{definition}
  \label{def:tetrahedron-flattening}
  A \defemph{flattening} of a shaped labeled tetrahedron \((\tau, \epsilon, z^0, z^1, z^2)\) is a tuple \((\zeta^0, \zeta^{1}, \zeta^{2})\) with
  of complex numbers
  \begin{equation}
    \label{eq:flattening-gluing-condition}
    \exp (\epsilon \zeta^{0}) = z^0, 
    \exp (\epsilon \zeta^{1}) = z^1, 
    \exp (\epsilon \zeta^{2}) = -z^2, 
    \text{ and }
    \sum_{k} \zeta^k = 0.
  \end{equation}
  The last condition means that \(\zeta^2 = -\zeta^0 - \zeta^1\), so we usually refer to a flattening by just \((\zeta^0, \zeta^1)\).

  A \defemph{flattening} of an ideal triangulation \(\mathcal{T} = \set{\tau_j}_j\) is a flattening of each \(\tau_j\) so that at every edge \(e\) of \(\mathcal{T}\)
  \begin{equation}
    \label{eq:edge-flattening-equation}
    \sum_{j} \zeta_j^{\nu(j)} = 0
  \end{equation}
  where the sum is over all tetrahedra \(\tau_j\) glued to edge \(e\) and \(\nu(j)\) is \(0\), \(1\), or \(2\) depending on the type of edge glued from \(\tau_j\).
  Observe that \eqref{eq:edge-flattening-equation} is the logarithm of \eqref{eq:edge-gluing-equation}.
\end{definition}
More generally, a flattening has \(\exp(\epsilon \zeta^{k}) = \pm z^{k}\); we can use a more restrictive definition because of the regularity of our triangulations.
This means that we only consider what \citeauthor{Neumann2004} \cite{Neumann2004} calls \emph{even} flattenings.

\begin{definition}
  A flattening can also be specified by a triple \((z; p^0, p^1)\), where \(z = z^0\) and the integers \((p^0, p^1)\) satisfy
  \[
    \zeta^0 = \log z + 2 \pi i p^0
    \text{ and }
    \zeta^1 = -\log(1-z) + 2 \pi i p^1
  \]
  In this context, it is natural to think of the space of flattenings as a Riemann surface over \(\CC \setminus \set{0,1}\).
  Specifically, consider the surface \(\CC_{\text{cut}}\) obtained by cutting \(\CC\setminus\{0,1\}\) along \((\infty, 0]\) and \([1,\infty)\).
  We can construct a \(\ZZ^2\)-cover \(\Sigma\) of \(\CC\setminus\{0,1\}\) by identifying the sheets of \(\CC_{\text{cut}} \times \ZZ^{2}\) via
  \begin{align*}
    (x + 0i; p^0,p^1) &\sim (x - 0i; p^0+1,p^1) & x &\in (-\infty, 0)
    \\
    (x + 0i; p^0,p^1) &\sim (x - 0i; p^0,p^1+1) & x &\in (1,\infty)
  \end{align*}
  We think of \(\Sigma\) as the space of flattened tetrahedra.
\end{definition}

\begin{definition}
  Fix the standard branch of the logarithm, with a branch cut from \(0\) to \(-\infty\) and arguments in \((-\pi, \pi]\).
  The \defemph{lifted dilogarithm} is the function \(\Sigma \to \CC / 2\pi^2 \ZZ\) given by
  \begin{equation}
    \label{eq:lifted-dilog-def}
    \dilr(z;p^{0},p^{1})
    \defeq
    R(z) - \frac{\pi^2}{6}
    +\frac{2\pi i}{2}( p^0 \log(1-z) + p^1 \log(z) )
  \end{equation}
  where
  \begin{equation}
    R(z) \defeq \dil(z) +  \frac{1}{2} \log(z) \log(1-z) 
  \end{equation}
  is the \defemph{Rogers dilogarithm} and
  \begin{equation}
    \dil(z) \defeq \int_{0}^{z} - \frac{\log(1-t)}{t} dt
  \end{equation}
  is the usual dilogarithm function.
  The imaginary part of \(R(z)\) is denoted \(D(z)\) in \cite{Zagier2007}.
\end{definition}

\begin{proposition}
  \(\dilr\) extends to a continuous function \(\dilr : \Sigma \to \CC/2\pi^2 \ZZ\).
\end{proposition}
\begin{proof}
  This is a consequence of \cite[Proposition 2.5]{Neumann2004}.
  In their notation our \((p^0, p^1)\) are their \((2p,2q)\), and we exclusively consider \emph{even} flattenings, so we get a continuous function modulo  \(2\pi^2\), not modulo \(\pi^2\).
\end{proof}

\begin{theorem}
  \label{thm:volume-neumann}
  Let \(L_0\) be a link in \(S^3\) and \(\rho_0 : \pi_1(S^{3} \setminus L) \to \slg\) a representation.
  Choose a flattened shaped triangulation \(\set{(\tau_j, \epsilon_j, \zeta_j^{0}, \zeta_{j}^{1})}_j\) of \(S^3 \setminus L_0\) in which all the tetrahedra are geometrically nondegenerate.
  Write \(\mathfrak{s}\) for the log-decoration on \(\extr{L_0}\) induced by the flattening.%
  \note{
    We describe how this works for the octahedral decomposition in \cref{def:induced-boundary-flattening}.
    Something similar works for a general triangulation, for example in \cite[Section 15]{Neumann2004}.
  }
  Suppose that \(\rho_0\) admits surgery along \(L_0\) to yield a \(\slg\)-manifold \((M, \rho)\), where \(\rho : \pi_1(M) \to \slg\) is induced by \(\rho_0\) by the surgery and \(\mathfrak{s}\) is compatible with the framing in the sense that
  \begin{equation}
    \label{eq:boundary-log-condition-Neumann14.7}
    p_j \mathfrak{s}(\mer_j) + q_j \mathfrak{s}(\lon_j) = 0
  \end{equation}
  for each component \(j\) with framing \(p_j/q_j\).
  The complex volume of \((M,\rho)\) is given by
   \[
     \cvol(M,\rho) = 
     -i \sum_{j} \epsilon_j \dilr(\zeta^{0}_j, \zeta^{1}_j) \in \CC/\pi^2 i \ZZ.
     \qedhere
  \]
\end{theorem}

\begin{proof}
  This is \cite[Theorem 14.7]{Neumann2004}.
  In particular, \cref{eq:boundary-log-condition-Neumann14.7} is the same as the condition that ``the log-parameter along a normal path in the neighborhood of a \(0\)-simplex that represents a filled cusp is zero if the path is null-homotopic in the added solid torus''.
\end{proof}

Later we will allow a more general boundary condition than \eqref{eq:boundary-log-condition-Neumann14.7}; this will require adding a correction term that we interpret as the complex volume of the added solid torus, equivalently the complex length of the added geodesic.
When \eqref{eq:boundary-log-condition-Neumann14.7} holds the correction lies in \(2 \pi^2 i \ZZ\) so we can ignore it.
For now, we give a useful lemma about the dependence on the flattening:

\begin{lemma}
  \label{lemma:flattening-dependence}
  As a function of the flattening,
  \begin{align*}
    L(z; p^0, p^1)
    =
    L(\zeta^0, \zeta^1)
    &=
    R(\exp(\zeta^0))
    -\frac{\pi^2}{6}
    +
    \frac{2\pi i }{2}
    \left( p^1 \zeta^0 - p^0 \zeta^1 \right)
    \\
    &=
    R(\exp(\zeta^0))
    -\frac{\pi^2}{6}
    +
    \frac{1}{2}
    \left( \zeta^0 \log(1-z) + \zeta^1 \log(z) \right)
    \\
    &\equiv
    R(\exp(\zeta^0))
    -\frac{\pi^2}{6}
    +
    \frac{1}{2}
    \left( - \zeta^0 \zeta^1 + \zeta^1 \log z + 2 \pi i  p^1 \log z \right)
    \pmod{2\pi^2 i \ZZ}
  \end{align*}
  and the dependence on the flattening is given by
  \begin{equation}
    \dilr( \zeta^0 + 2\pi i k^0, \zeta^1 + 2\pi i k^1 )
    \equiv
    \dilr( \zeta^0, \zeta^1)
    +
    \frac{2\pi i}{2}\left( k^1  \zeta^0 - k^0 \zeta^1\right)
    \pmod{2\pi^2 \ZZ}.
    \qedhere
  \end{equation}
\end{lemma}
\begin{proof}
  The first relation comes from substituting \(\log(z) = 2\pi i p^0 - \zeta^0\) and \(\log(1-z) = 2\pi i p^1 - \zeta^1\) into the definition of \(\dilr\), and then the second is immediate.
\end{proof}

\subsection{Shaped tangles and the octahedral decomposition}
We can now describe how to produce ideal triangulations from link diagrams.
\begin{definition}
  Let \(L\) be a link in \(S^3\) with \(n\) components and let \(D\) be an oriented diagram of \(L\).
  Thinking of \(D\) as a decorated \(4\)-valent graph \(G\) embedded in \(S^2\), the \defemph{segments} of \(D\) are the edges%
  \note{%
    Usually these are called the ``edges'' of the diagram, but we do not want to confuse them with edges of ideal polyhedra.
  }
  of \(G\).
  A \defemph{region} of a diagram is a connected component of the complement of \(G\), equivalently a vertex of the dual graph of \(G\).
  For example, Figure \ref{fig:figure-eight-labeled} shows an (oriented) diagram with the segments labeled.
\end{definition}
\begin{marginfigure}
\begingroup%
  \makeatletter%
  \providecommand\color[2][]{%
    \errmessage{(Inkscape) Color is used for the text in Inkscape, but the package 'color.sty' is not loaded}%
    \renewcommand\color[2][]{}%
  }%
  \providecommand\transparent[1]{%
    \errmessage{(Inkscape) Transparency is used (non-zero) for the text in Inkscape, but the package 'transparent.sty' is not loaded}%
    \renewcommand\transparent[1]{}%
  }%
  \providecommand\rotatebox[2]{#2}%
  \newcommand*\fsize{\dimexpr\f@size pt\relax}%
  \newcommand*\lineheight[1]{\fontsize{\fsize}{#1\fsize}\selectfont}%
  \ifx\svgwidth\undefined%
    \setlength{\unitlength}{132.93800354bp}%
    \ifx\svgscale\undefined%
      \relax%
    \else%
      \setlength{\unitlength}{\unitlength * \real{\svgscale}}%
    \fi%
  \else%
    \setlength{\unitlength}{\svgwidth}%
  \fi%
  \global\let\svgwidth\undefined%
  \global\let\svgscale\undefined%
  \makeatother%
  \begin{picture}(1,0.73142784)%
    \lineheight{1}%
    \setlength\tabcolsep{0pt}%
    \put(0,0){\includegraphics[width=\unitlength,page=1]{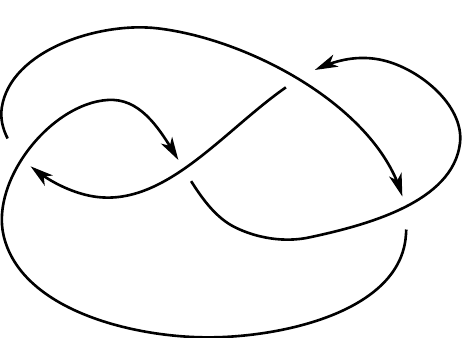}}%
    \put(0.18574584,0.53607107){\makebox(0,0)[lt]{\lineheight{1.25}\smash{\begin{tabular}[t]{l}$1$\end{tabular}}}}%
    \put(0.49378445,0.15159342){\makebox(0,0)[lt]{\lineheight{1.25}\smash{\begin{tabular}[t]{l}$2$\end{tabular}}}}%
    \put(0.86669712,0.6034901){\makebox(0,0)[lt]{\lineheight{1.25}\smash{\begin{tabular}[t]{l}$3$\end{tabular}}}}%
    \put(0.51156609,0.41261917){\makebox(0,0)[lt]{\lineheight{1.25}\smash{\begin{tabular}[t]{l}$4$\end{tabular}}}}%
    \put(0.13125042,0.24642853){\makebox(0,0)[lt]{\lineheight{1.25}\smash{\begin{tabular}[t]{l}$5$\end{tabular}}}}%
    \put(0.28308073,0.69587022){\makebox(0,0)[lt]{\lineheight{1.25}\smash{\begin{tabular}[t]{l}$6$\end{tabular}}}}%
    \put(0.82184746,0.44367152){\makebox(0,0)[lt]{\lineheight{1.25}\smash{\begin{tabular}[t]{l}$7$\end{tabular}}}}%
    \put(0.32516562,0.03348949){\makebox(0,0)[lt]{\lineheight{1.25}\smash{\begin{tabular}[t]{l}$8$\end{tabular}}}}%
  \end{picture}%
\endgroup%

  \caption{A diagram of the figure-eight knot, with the \(8\) segments indexed by \(1, \dots, 8\).}
  \label{fig:figure-eight-labeled}
\end{marginfigure}
In an oriented diagram all crossings are positive or negative, as shown in Figure \ref{fig:crossing-types}.
Our preference is to read crossings left-to-right.
\begin{figure}
  \centering
\begingroup%
  \makeatletter%
  \providecommand\color[2][]{%
    \errmessage{(Inkscape) Color is used for the text in Inkscape, but the package 'color.sty' is not loaded}%
    \renewcommand\color[2][]{}%
  }%
  \providecommand\transparent[1]{%
    \errmessage{(Inkscape) Transparency is used (non-zero) for the text in Inkscape, but the package 'transparent.sty' is not loaded}%
    \renewcommand\transparent[1]{}%
  }%
  \providecommand\rotatebox[2]{#2}%
  \newcommand*\fsize{\dimexpr\f@size pt\relax}%
  \newcommand*\lineheight[1]{\fontsize{\fsize}{#1\fsize}\selectfont}%
  \ifx\svgwidth\undefined%
    \setlength{\unitlength}{283.55544662bp}%
    \ifx\svgscale\undefined%
      \relax%
    \else%
      \setlength{\unitlength}{\unitlength * \real{\svgscale}}%
    \fi%
  \else%
    \setlength{\unitlength}{\svgwidth}%
  \fi%
  \global\let\svgwidth\undefined%
  \global\let\svgscale\undefined%
  \makeatother%
  \begin{picture}(1,0.27643014)%
    \lineheight{1}%
    \setlength\tabcolsep{0pt}%
    \put(0,0){\includegraphics[width=\unitlength,page=1]{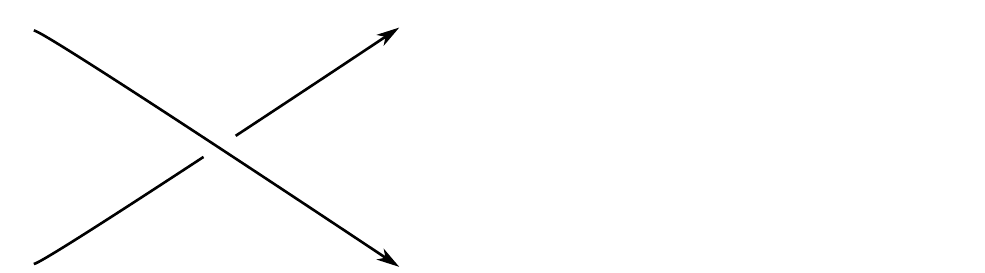}}%
    \put(-0.00351287,0.24426852){\makebox(0,0)[lt]{\lineheight{1.25}\smash{\begin{tabular}[t]{l}$1$\end{tabular}}}}%
    \put(-0.00351287,0.00621985){\makebox(0,0)[lt]{\lineheight{1.25}\smash{\begin{tabular}[t]{l}$2$\end{tabular}}}}%
    \put(0.41174984,0.24426852){\makebox(0,0)[lt]{\lineheight{1.25}\smash{\begin{tabular}[t]{l}$2'$\end{tabular}}}}%
    \put(0.41174984,0.00621985){\makebox(0,0)[lt]{\lineheight{1.25}\smash{\begin{tabular}[t]{l}$1'$\end{tabular}}}}%
    \put(0,0){\includegraphics[width=\unitlength,page=2]{crossing-types.pdf}}%
    \put(0.56563205,0.24427706){\makebox(0,0)[lt]{\lineheight{1.25}\smash{\begin{tabular}[t]{l}$1$\end{tabular}}}}%
    \put(0.56563205,0.00622839){\makebox(0,0)[lt]{\lineheight{1.25}\smash{\begin{tabular}[t]{l}$2$\end{tabular}}}}%
    \put(0.9808948,0.24427706){\makebox(0,0)[lt]{\lineheight{1.25}\smash{\begin{tabular}[t]{l}$2'$\end{tabular}}}}%
    \put(0.9808948,0.00622839){\makebox(0,0)[lt]{\lineheight{1.25}\smash{\begin{tabular}[t]{l}$1'$\end{tabular}}}}%
  \end{picture}%
\endgroup%

  \caption{Positive (left) and negative (right) crossings.}
  \label{fig:crossing-types}
\end{figure}
As shown there, we usually refer to the segments at a given crossing by \(1\), \(2\), \(1'\), and \(2'\).
We similarly refer to the regions touching the crossing as \(N\), \(S\), \(E\), and \(W\).
The labeling conventions are summarized in Figure \ref{fig:crossing-regions}.
\begin{marginfigure}
\begingroup%
  \makeatletter%
  \providecommand\color[2][]{%
    \errmessage{(Inkscape) Color is used for the text in Inkscape, but the package 'color.sty' is not loaded}%
    \renewcommand\color[2][]{}%
  }%
  \providecommand\transparent[1]{%
    \errmessage{(Inkscape) Transparency is used (non-zero) for the text in Inkscape, but the package 'transparent.sty' is not loaded}%
    \renewcommand\transparent[1]{}%
  }%
  \providecommand\rotatebox[2]{#2}%
  \newcommand*\fsize{\dimexpr\f@size pt\relax}%
  \newcommand*\lineheight[1]{\fontsize{\fsize}{#1\fsize}\selectfont}%
  \ifx\svgwidth\undefined%
    \setlength{\unitlength}{113.47111416bp}%
    \ifx\svgscale\undefined%
      \relax%
    \else%
      \setlength{\unitlength}{\unitlength * \real{\svgscale}}%
    \fi%
  \else%
    \setlength{\unitlength}{\svgwidth}%
  \fi%
  \global\let\svgwidth\undefined%
  \global\let\svgscale\undefined%
  \makeatother%
  \begin{picture}(1,0.700154)%
    \lineheight{1}%
    \setlength\tabcolsep{0pt}%
    \put(0,0){\includegraphics[width=\unitlength,page=1]{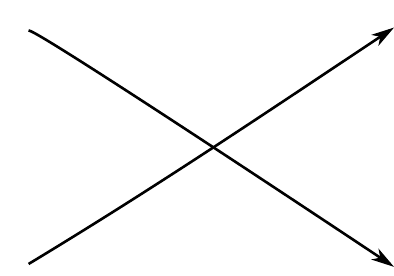}}%
    \put(-0.02199761,0.61980597){\makebox(0,0)[lt]{\lineheight{1.25}\smash{\begin{tabular}[t]{l}$1$\end{tabular}}}}%
    \put(-0.02199761,0.02494095){\makebox(0,0)[lt]{\lineheight{1.25}\smash{\begin{tabular}[t]{l}$2$\end{tabular}}}}%
    \put(1.01571146,0.61980597){\makebox(0,0)[lt]{\lineheight{1.25}\smash{\begin{tabular}[t]{l}$2'$\end{tabular}}}}%
    \put(1.01571146,0.02494095){\makebox(0,0)[lt]{\lineheight{1.25}\smash{\begin{tabular}[t]{l}$1'$\end{tabular}}}}%
    \put(0.48694249,0.45456569){\makebox(0,0)[lt]{\lineheight{1.25}\smash{\begin{tabular}[t]{l}$N$\end{tabular}}}}%
    \put(0.48694249,0.19018121){\makebox(0,0)[lt]{\lineheight{1.25}\smash{\begin{tabular}[t]{l}$S$\end{tabular}}}}%
    \put(0.68523089,0.28932538){\makebox(0,0)[lt]{\lineheight{1.25}\smash{\begin{tabular}[t]{l}$E$\end{tabular}}}}%
    \put(0.28865412,0.28932538){\makebox(0,0)[lt]{\lineheight{1.25}\smash{\begin{tabular}[t]{l}$W$\end{tabular}}}}%
  \end{picture}%
\endgroup%

  \caption{Regions near a crossing.}
  \label{fig:crossing-regions}
\end{marginfigure}

\begin{definition}
  A \defemph{shape} is a triple of nonzero complex numbers.
  We usually denote a shape by $\chi = (a, b, m) \in (\CC \setminus \{0\})^{3}$, and when it is assigned to a segment $i$ of a tangle diagram we denote it $\chi_i = (a_i, b_i, m_i)$.
\end{definition}

\begin{definition}
  \label{def:braiding}
  The \defemph{braiding} $B$ is the partially-defined map given by $B(\chi_1, \chi_2) = (\chi_{2'}, \chi_{1'})$, where
  \allowdisplaybreaks{}
  \begin{gather}
    \label{eq:a-transf-positive}
    \begin{aligned}
      a_{1'}
      &=
      a_1 A^{-1}
      \\
      a_{2'}
      &=
      a_2 A
      \\
      A &= 1 - \frac{m_1 b_1}{b_2} \left(1 - \frac{a_1}{m_1}\right)\left(1 - \frac{1}{m_2 a_2}\right)
    \end{aligned}
    \\
    \label{eq:b-transf-positive}
    \begin{aligned}
      b_{1'}
      &=
      \frac{m_2 b_2}{m_1}
      \left(
        1 - m_2 a_2 \left( 1 - \frac{b_2}{m_1 b_1} \right)
      \right)^{-1}
      \\
      b_{2'}
      &=
      b_1
      \left(
        1 - \frac{m_1}{a_1}\left( 1 - \frac{b_2}{m_1 b_1} \right)
      \right)
    \end{aligned}
    \\
    \label{eq:m-transf-positive}
    \begin{aligned}
      m_{1'}
      &= m_1
      &
      m_{2'}
      &= m_2
    \end{aligned}
  \end{gather}
  We think of $B$ as being associated to a positive crossing with incoming strands $1$ and $2$ and outgoing strands $2'$ and $1' $, as in Figure \ref{fig:crossing-regions}.
  The map $B$ is generically invertible, and if $(\chi_{2'}, \chi_{1'}) = B^{-1}(\chi_1, \chi_2)$, then
  \begin{gather}
    \label{eq:a-transf-negative}
    \begin{aligned}
      a_{1'}
      &=
      a_1 \tilde A^{-1}
      \\
      a_{2'}
      &=
      a_2 \tilde A
      \\
      \tilde A
      &=
      1 - \frac{b_2}{m_1 b_1}\left(1 - m_1 a_1 \right)\left(1 - \frac{m_2}{a_2}\right).
    \end{aligned}
    \\
    \label{eq:b-transf-negative}
    \begin{aligned}
      b_{1'}
      &=
      \frac{m_2 b_2}{m_1}
      \left(
        1 - \frac{a_2}{m_2} \left( 1 - \frac{m_1 b_1}{b_2} \right)
      \right)
      \\
      b_{2'}
      &=
      b_1
      \left(
        1 - \frac{1}{m_1 a_1} \left( 1 - \frac{m_1 b_1}{b_2} \right)
      \right)^{-1}
    \end{aligned}
    \\
    \label{eq:m-transf-negative}
    \begin{aligned}
      m_{1'}
      &= m_1
      &
      m_{2'}
      &= m_2
    \end{aligned}
  \end{gather}
\end{definition}

\begin{definition}
  \label{def:shaping}
  We say that a tangle diagram $D$ is \defemph{shaped} if its segments are assigned shapes $\{\chi_i\}$ so that at each positive crossing (labeled as in Figure \ref{fig:crossing-regions}) we have $B(\chi_1, \chi_2) = (\chi_{2'}, \chi_{1'})$, and similarly for negative crossings and $B^{-1}$.
  Part of this requirement is that all the components of \(\chi_{1'}\) and \(\chi_{2'}\) lie in \(\CC^{\times}\).
  For example, this means that at a positive crossing we must assign \(\chi_1\) and \(\chi_2\) so that
  \[
    A = 
      1 - \frac{b_2}{m_1 b_1}\left(1 - m_1 a_1 \right)\left(1 - \frac{m_2}{a_2}\right)
  \]
  is not \(0\) or \(\infty\).
\end{definition}

We now explain how these relate to hyperbolic geometry via the \defemph{octahedral decomposition}, which
 assigns an ideal triangulation to any link diagram by putting a twisted ideal octahedron at each crossing.
The gluing pattern for the faces and edges of the octahedra is determined by the combinatorics of the diagram; see \cite[Section 3]{McPhailSnyder2022} for details.
We consider the version where the octahedral are decomposed into four tetrahedra, as in \cref{fig:four-term}.

\begin{figure}
  \centering
  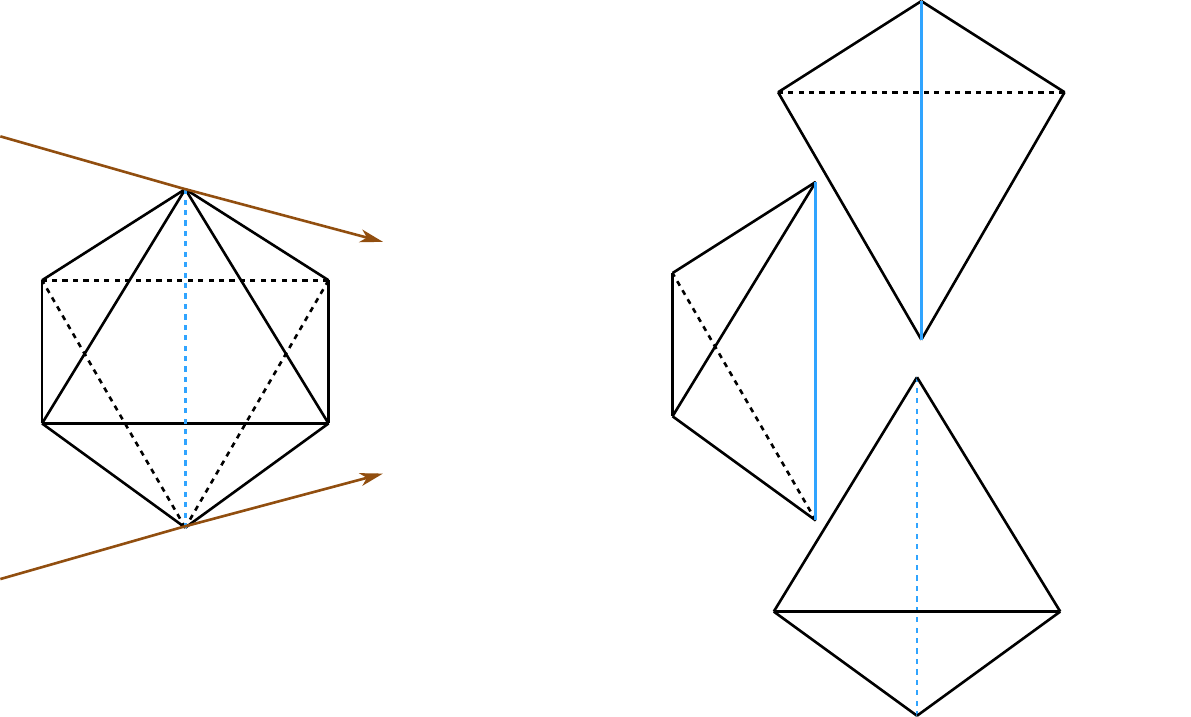
  \caption{The four-term decomposition of an ideal octahedron at a positive crossing.}
  \label{fig:four-term}
\end{figure}
\begin{table}
  \centering
  \[
    \begin{array}{c|c|ccc}
      & \text{vertices} & \text{orientation } & \text{shape } z^{0} 
      \\
      \hline
      \tau_N & P_{2}  P_{1}  P_{-}  P_{+}' & \epsilon & b_{2'}/b_1 
      \\
      \tau_W & P_{2}  P_{1}  P_{-}  P_{+} & -\epsilon & m_1 b_1/b_2
      \\
      \tau_S & P_{2}  P_{1}  P_{-}'  P_{+} & \epsilon & m_2 b_2/m_1 b_{1'}
      \\
      \tau_E & P_{2}  P_{1}  P_{-}'  P_{+}' & -\epsilon & b_{1'}/m_2 b_{2'}
      \\
    \end{array}
  \]
  \caption{Geometric data for the tetahedra at a crossing of sign \(\epsilon\).}
  \label{table:simplex-data}
\end{table}

\begin{theorem}
  \label{thm:shape-paper-results}
  Let \(L\) be a link in \(S^3\) and \(D\) an oriented diagram of \(L\).
  \begin{enumerate}
    \item  Any shaping of  \(D\) gives a well-defined \defemph{holonomy representation} \(\rho : \pi_1(S^3 \setminus L) \to \slg\).
    \item Any representation \(\rho : \pi_1(S^3\setminus L) \to \slg\) is conjugate to a representation \(\rho'\) that can be expressed as the holonomy of a shaping of \(D\).
    \item A shaping of a tangle diagram \(D\) gives a solution to the gluing equations of its octahedral decomposition via the shapes of \cref{table:simplex-data}.
      The holonomy of the induced hyperbolic structure is the representation \(\rho\).
    \item The shaping induces a decoration \(\delta\) of \(\rho\) whose homomorphism is given by
  \begin{align}
    \label{eq:meridian-eigenvalue}
    \delta(\mer_j) &= m_j
    \\
    \label{eq:longitude-eigenvalue}
    \delta(\lon_j) &= m_j^{-w_j} \prod_{k} b_k^{\eta_k}
  \end{align}
  where \(w_j\) is the writhe of component \(j\), the product is over all segments in component \(j\), and
  \[
    \eta_k
    \defeq
    \begin{cases}
      1 & \text{if segment \(k\) is over-under,}
      \\
      -1 & \text{if it is  under-over, and}
      \\
      0 & \text{otherwise.}
    \end{cases}
    \qedhere
  \]
  \end{enumerate}
\end{theorem}
\begin{proof}
  The claims are
  \begin{enumerate}
    \item Theorem 2.17,
    \item Theorem 2.18,
    \item Theorems 3.5 and 3.6, and
    \item Theorem 4.4
  \end{enumerate}
  of \cite{McPhailSnyder2022}.
\end{proof}

Certain choices of shaping give geometrically degenerate tetrahedra; later we will see that we can still make sense of \(\volf\) in this case.

\begin{proposition}
  \label{thm:pinched-def}
  Let \(c\) be a crossing of a shaped diagram.
  If one of the relations
  \begin{equation}
    b_{2'} = b_1, \quad
    b_{2} = m_1 b_1, \quad
    m_2 b_{2} = m_1 b_{1'}, \quad
    m_2 b_{2'} = b_{1'}, \quad
  \end{equation}
  holds then all four of them do.
  In this case we say that the crossing is \defemph{pinched}, and in this case the four tetrahedra at the crossing are geometrically degenerate.
\end{proposition}
\begin{proof}
  Once you know what to look for this is a simple check using \(B\) and \(B^{-1}\).
\end{proof}

One advantage of using the shape coordinates for the octahedral decomposition is that they easily determine flattenings: all we have to do is pick logarithms of the shape coordinates (and some related parameters).
We abuse terminology and also call this choice of logarithms a ``flattening''.

\begin{definition}
  Let \((D,\chi)\) be a shaped diagram.
  A \defemph{flattening} \(\mathfrak f\) is a choice of logarithms of a number of variables related to the shaping \(\chi\).
  Specifically:
  \begin{itemize}
    \item Each component of \(D\) has an associated meridian eigenvalue \(m_j\).
      We choose a logarithm \(\mu_j\) with \(e^{2\pi i \mu_j} = m_j\), where for convenience we pull out a factor of \(2\pi i\).
    \item Each segment of \(D\) has a \(b\)-variable \(b_j\).
      We choose a logarithm \(\beta_j\) with \(e^{2\pi i \beta_j} = b_j\).
    \item Each region of \(D\) has a region log-parameter \(\gamma_j\) related to the \(a\)-variables as in  \cref{fig:region-variable-rule-log}.\note{
        In \cite[Section 5.2]{McPhailSnyder2022} we considered region variables \(r_j\) whose ratios give the \(a\)-variables.
        The \(\gamma_j\) are just a choices of logarithm of the \(r_j\).
      }
    \item At each crossing \(c_j\) of \(D\) with sign \(\epsilon\), consider the parameter
    \begin{equation*}
      K
      = 
      \frac{\exp(\gamma_N)}{1 - \left(b_{2'}/b_{1}\right)^{\epsilon}}
    \end{equation*}
      related to the shape of the internal edge of the octahedron at \(c_j\).
      Choose \(\kappa_i\) with \(e^{2\pi i\kappa_j} = K\).
      If the crossing \(c_j\) is pinched then \(K_j = \infty\) so this definition does not make sense and we omit the choice of \(\kappa_j\).
      We will later see we can still make sense of the complex volume of such crossings.
  \end{itemize}
  We call \((D, \chi, \mathfrak f)\) a \defemph{flattened diagram}; this implies a choice of shaping.
\end{definition}

\begin{marginfigure}
  \centering
\begingroup%
  \makeatletter%
  \providecommand\color[2][]{%
    \errmessage{(Inkscape) Color is used for the text in Inkscape, but the package 'color.sty' is not loaded}%
    \renewcommand\color[2][]{}%
  }%
  \providecommand\transparent[1]{%
    \errmessage{(Inkscape) Transparency is used (non-zero) for the text in Inkscape, but the package 'transparent.sty' is not loaded}%
    \renewcommand\transparent[1]{}%
  }%
  \providecommand\rotatebox[2]{#2}%
  \newcommand*\fsize{\dimexpr\f@size pt\relax}%
  \newcommand*\lineheight[1]{\fontsize{\fsize}{#1\fsize}\selectfont}%
  \ifx\svgwidth\undefined%
    \setlength{\unitlength}{109.92816067bp}%
    \ifx\svgscale\undefined%
      \relax%
    \else%
      \setlength{\unitlength}{\unitlength * \real{\svgscale}}%
    \fi%
  \else%
    \setlength{\unitlength}{\svgwidth}%
  \fi%
  \global\let\svgwidth\undefined%
  \global\let\svgscale\undefined%
  \makeatother%
  \begin{picture}(1,0.3133619)%
    \lineheight{1}%
    \setlength\tabcolsep{0pt}%
    \put(0,0){\includegraphics[width=\unitlength,page=1]{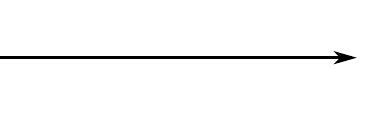}}%
    \put(0.341132,0.23042404){\makebox(0,0)[lt]{\lineheight{1.25}\smash{\begin{tabular}[t]{l}$\gamma_N$\end{tabular}}}}%
    \put(0.341132,0.02574484){\makebox(0,0)[lt]{\lineheight{1.25}\smash{\begin{tabular}[t]{l}$\gamma_S$\end{tabular}}}}%
    \put(0.95516958,0.16219765){\makebox(0,0)[lt]{\lineheight{1.25}\smash{\begin{tabular}[t]{l}$a_j$\end{tabular}}}}%
  \end{picture}%
\endgroup%

  \caption{The parameters for regions adjacent across a strand with \(a\)-variable \(a_j\) are required to satisfy \(e^{2\pi i(\gamma_{S}-\gamma_{N})} = a_j\).}
  \label{fig:region-variable-rule-log}
\end{marginfigure}

\begin{theorem}
  \label{thm:our-flattening}
  Let \((D, \chi, \mathfrak f)\) be a flattened diagram.
  At each crossing with sign \(\epsilon\), assign the tetrahedra the log-parameters
  \begin{align*}
    \zeta_N^0
    &=
    2\pi i \epsilon (\beta_{2'} - \beta_1)
    &
    \zeta_N^1
    &=
    2\pi i( \kappa - \gamma_N)
    \\
    \zeta_W^0
    &=
    2\pi i \epsilon(\beta_{2} - \beta_{1} - \mu_1 )
    &
    \zeta_W^1
    &=
    2\pi i( \kappa - \gamma_W + \epsilon \mu_1)
    \\
    \zeta_S^0
    &=
    2\pi i \epsilon(\beta_{2} - \beta_{1'} + \mu_2 - \mu_1 )
    &
    \zeta_S^1
    &=
    2\pi i (\kappa - \gamma_S + \epsilon (\mu_1 - \mu_2))
    \\
    \zeta_E^0
    &=
    2\pi i \epsilon(\beta_{2'} - \beta_{1'} + \mu_2 )
    &
    \zeta_E^1
    &=
    2\pi i( \kappa - \gamma_E - \epsilon \mu_2)
  \end{align*}
  This gives a flattening of the four-term octahedral decomposition associated to \(D\).
\end{theorem}
\begin{proof}
  It is obvious that \(e^{\zeta_N^0} = z_N^0\), and similarly for the other tetrahedra.
  Checking that \(e^{\zeta_N^1} = z_N^1\) is elementary given the right identities \cite[Lemma 2.10]{McPhailSnyder2022} on the shapes.
  Finally, we can check \cref{eq:flattening-gluing-condition} by following the same procedure as in the proof of \cite[Theorem 3.2]{McPhailSnyder2022}:
  it is obvious that it holds for the vertical edges in each region, and the horizontal edges follow from a logarithmic version of the \(m\)-hyperbolicity equations.
\end{proof}

In general a choice of flattening is a lot of data.
However, it turns out that all that matters for the value of \(\volf\) is the log-decoration induced by it.

\begin{definition}
  \label{def:induced-boundary-flattening}
  Let \((D, \chi, \mathfrak{f})\) be a flattened, shaped diagram of a link \(L\).
  For each component \(L_j\) of \(L\), set
  \begin{equation}
    \mathfrak{s}(\mer_j) = \mu_j
  \end{equation}
  where \(\mu_j\) is the log-meridian of any segment of \(L_j\), and
  \begin{equation}
    \label{eq:log-longitude-sum}
    \mathfrak{s}(\lon_j) = -w_j \mu_j +  \sum_{k} (-1)^{\eta_k} \beta_k
  \end{equation}
  where \(w_j\) is the writhe of component \(j\), the sum is over all segments in component \(j\), and
  \[
    \eta_k
    \defeq
    \begin{cases}
      1 & \text{if segment \(k\) is over-under,}
      \\
      -1 & \text{if it is  under-over, and}
      \\
      0 & \text{otherwise.}
    \end{cases}
  \]
  By part 4 of \cref{thm:shape-paper-results} \(\mathfrak{s}\) is a log-decoration, and we call it the log-decoration \defemph{induced} by \(\mathfrak{f}\).
\end{definition}

\begin{remark}
  \label{rem:2-categories}
  Shaped tangle diagrams form a monoidal category in the usual way: we can compose tangles (in our conventions, horizontally) when they have matching boundary points, including orientations and shapes.
  There is also a tensor product given by disjoint union, which in our conventions is vertical composition.
  
  Flattened tangle diagrams have additional parameters on the regions, so they form a \(2\)-category: we can only vertically or horizontally compose diagrams when their region log-parameters match on the identified regions.
  We refer to \cite[Section 2.2]{Lauda2012} for more on the graphical calculus of \(2\)-categories.
  We mostly focus on the behaviour of \(\volf\) on link diagrams, so we do not emphasize the \(2\)-category structure, but it would be interesting to explore it further.
\end{remark}

\section{The complex volume of a link exterior}
In this section, given
\begin{itemize}
  \item an oriented link  \(L\) in  \(S^3\),
  \item a  representation \(\rho : \pi_1(S^3 \setminus L) \to \slg\), and
  \item a log-decoration \(\mathfrak{s} \in \logrepv(\extr L, \rho)\) compatible with both of these,
\end{itemize}
we define a complex number
\[
  \volf(L, \rho)(\mathfrak{s}) \in \CC/2\pi^2 i \ZZ
\]
which we think of as the complex volume of the link exterior \(\extr L\).
Given a diagram \(D\) of \(L\), a shaping \(\chi\) of \(D\) corresponding to \(\rho\), and a flattening \(\mathfrak{f}\) corresponding to \(\mathfrak{s}\), \(\volf\) is defined by
\[
  \volf(L, \rho)(\mathfrak{s}) = \volf(D, \chi)(\mathfrak{f}) \defeq \sum_{c} \volf(c, \mathfrak{f})
\]
where the sum is over crossings \(c\) of \(D\) and the volume \(\volf(c, \mathfrak{f})\) of a crossing is the sum
\[
  \volf(c, \mathfrak{f})
  =
  -i\epsilon \left[
    \dilr(\zeta_N^0, \zeta_N^1)
    -\dilr(\zeta_W^0, \zeta_W^1)
    +\dilr(\zeta_S^0, \zeta_S^1)
    -\dilr(\zeta_E^0, \zeta_E^1)
  \right]
\]
of the volumes of the tetrahedra at that crossing computed using the log-parameters  \(\zeta_j^k\) determined by the flattening \(\mathfrak{f}\).

We then show that
\begin{itemize}
  \item \(\volf(D, \chi)(\mathfrak{f})\) depends only on the boundary \(\log\)-structure \(\mathfrak{s}\) determined by \(\mathfrak{f}\), but not the rest of \(\mathfrak{f}\), and
  \item \(\volf(D, \chi)(\mathfrak{f})\) is invariant under Reidemeister moves.
\end{itemize}
These establish that \(\volf\) really is an invariant of \((L, \rho, \mathfrak{s})\) and not superficial choices like the underlying diagram or the choice of flattening.
Once this is done, we show that
\begin{itemize}
  \item \(\volf(L, \rho)(\mathfrak{s})\) does not depend on the chosen orientation of \(L\), and
  \item \(\volf(L, \rho)(\mathfrak{s}) = \volf(L, \rho')(\mathfrak{s})\) whenever \(\rho\) and \(\rho'\) are conjugate.
\end{itemize}

Our perspective on \(\volf\) comes from quantum invariants of links, in particular the holonomy invariants of \cite{Blanchet2018,McPhailSnyderThesis}.
We will see that \(\volf(T, \rho, \mathfrak{s})\) also makes sense for tangles \(T\), not just links, and that it obeys cut-and-paste decomposition rules analogous to those for quantum link invariants.
The proof that \(\volf\) is conjugation-invariant uses the formalism of \cite{Blanchet2018}, which was originally motivated by quantum holonomy invariants.

\subsection{Defining \texorpdfstring{\(\volf\)}{𝒱}}

\begin{definition}
  \label{def:crossing-volume}
  Let \((D, \chi, \mathfrak{f})\) be a flattened link diagram, and let \(c\) be a crossing of \(D\) of sign \(\epsilon\).
  If \(c\) is not pinched, the \defemph{volume} of \(c\) is
  \begin{equation}
    \label{eq:crossing-volume}
    \volf(c, \mathfrak f) \defeq
    -i\epsilon \left[
      \dilr(\zeta_N^0, \zeta_N^1)
      -\dilr(\zeta_W^0, \zeta_W^1)
      +\dilr(\zeta_S^0, \zeta_S^1)
      -\dilr(\zeta_E^0, \zeta_E^1)
    \right]
  \end{equation}
  where the \(\zeta_j^k\) are computed using the log-parameters of \(\mathfrak{f}\).
  If \(c\) is pinched, we instead define
  \begin{equation}
    \label{eq:crossing-volume-pinched}
    \volf(c, \mathfrak{f})
    \defeq
    2 \pi^2 i
    \left[
      \begin{aligned}
        &\phantom{+}\beta_1 (\gamma_W - \gamma_N) - \beta_{1'}(\gamma_S - \gamma_E)
        \\
        &+\beta_2 (\gamma_S - \gamma_W) - \beta_{2'}(\gamma_E - \gamma_N)
        \\
        &- \mu_1(\epsilon (\beta_1 - \beta_{1'} + \mu_2) + \gamma_S - \gamma_W)
        \\
        &- \mu_2(\epsilon (\beta_{2'} - \beta_{2} + \mu_1) + \gamma_E - \gamma_S)
      \end{aligned}
    \right]
  \end{equation}
  As shown in \cref{lemma:volf-continuous-pinched}, \eqref{eq:crossing-volume-pinched} is obtained by taking the limit of \eqref{eq:crossing-volume} as \(\exp(\zeta_k^0) \to 1\) in the appropriate way.

  The \defemph{volume} of the diagram \(D\) is the sum of the volumes of the crossings:
  \begin{equation}
    \label{eq:diagram-volume}
    \volf(D, \chi, \mathfrak{f})
    \defeq 
    \sum_{c} \volf(c, \mathfrak{f}).
  \end{equation}
  In each case we think of \(\volf\) as taking values in \(\CC/2\pi^2 i \ZZ\).
\end{definition}

Our first result is that \(\volf\) depends only on simpler data than the flattening:
\begin{theorem}
  \label{thm:diagram-flattening-dependence}
  Let \((D,\chi)\)  be a shaped link diagram of a link \((L, \rho)\).
  Then \(\volf(D, \chi, \mathfrak{f})\) depends only on the log-decoration induced by \(\mathfrak{f}\), so we can think of it as a function
  \[
    \volf(D, \chi) : \logrepv(L, \rho) \to \CC/2\pi^2 i \ZZ
  \]
  on the space of log-decorations of \((L, \rho)\).
  Furthermore, the dependence on the log-decoration is given by
  \begin{equation}
    \label{eq:cvol-s-dependence}
    \volf(D, \chi)(\mathfrak{s}')
    -
    \volf(D, \chi)(\mathfrak{s})
    \equiv
    4\pi^2i
    \sum_{j}
    \Delta \lambda_j
    \mathfrak{s}(\mer_j)
    -
    \Delta \mu_j
    \mathfrak{s}(\lon_j)
    \pmod{2\pi^2 i \ZZ}
  \end{equation}
  where the sum is over connected components of \(D\) (that is, over boundary components of \(\extr L\)) and
  \[
    \Delta \lambda_j
    =
    \mathfrak{s}'(\lon_j) - \mathfrak{s}(\lon_j)
    , \quad
    \Delta \mu_j
    =
    \mathfrak{s}'(\mer_j) - \mathfrak{s}(\mer_j).\qedhere
  \]
\end{theorem}

\begin{remark}
  This result is false for tangle diagrams: in addition to the induced log-decoration \(\mathfrak{s}\), the volume of a tangle diagram can also depend on the region log-parameters \(\gamma_j\) assigned to open regions.
  In a link diagram, there are no open regions, so this does not come up.

  Geometrically interpreting the dependence on the region parameters could lead to a better understanding of the \(2\)-category structure on flattened tangles discussed in \cref{rem:2-categories}.
\end{remark}

\begin{lemma}
  \label{lemma:kappa-independent}
  The volume \(\volf(c, \mathfrak{f})\) at a  crossing does not depend on the value of \(\kappa\).
\end{lemma}
\begin{proof}
  Since \(\kappa\) does not appear in \cref{eq:crossing-volume-pinched} we only need to consider the case that \(c\) is not pinched.
  Suppose we use a different value \(\kappa'\), which has \(\kappa' = \kappa + k\) for some \(k \in \ZZ\).
  Because \(\kappa\) appears only in the \(\zeta_j^1\) parameters, not the \(\zeta_j^0\), by \cref{lemma:flattening-dependence} we have
  \[
    \volf(c, \mathfrak{f}')
    -
    \volf(c, \mathfrak{f}')
    \equiv
    \frac{1}{2} 2\pi i k \epsilon \left( \zeta_N^0 - \zeta_W^0 + \zeta_S^0 - \zeta_E^0 \right)
    = 0
    .
    \qedhere
  \]
\end{proof}
\begin{remark}
  Because \(\volf\) never depends on the choice of logarithms \(\kappa\) we will drop them from our discussion of flattenings going forward.
  We still need to use them to compute \(\volf\), but we can always set  \(\kappa = \log K\), so there's no point mentioning this every time.
  Another reason is to make it simpler to discuss pinched crossings: at these crossings \(\kappa_j = \log 0 + 2\pi i k\) is not well-defined.
  The next lemma says this is not an issue.
\end{remark}

\begin{lemma}
  \label{lemma:volf-continuous-pinched}
  For any diagram \(D\), let \(\mathfrak{F}(D)\) be the space of flattenings \(\mathfrak{f}\) of \(D\) in which we set \(\kappa = \log K\) at every non-pinched crossing.
  Then \(\volf\) is a continuous function on \(\mathfrak{F}(D)\) with values in \(\CC / 2\pi^2 i \ZZ\).
\end{lemma}
\begin{proof}
  This is obvious when all the crossings are pinched, so it suffices to check that \(\volf(c, \mathfrak{f})\) is continuous in the limit where \(c\) becomes pinched.
  Each argument \((\zeta_j^0)^{\epsilon_j}\) goes to \(1\) in this limit, so \(R((\zeta_j^0)^{\epsilon_j})\) goes to \(0\) as well.
  We only need to worry about the correction terms, and the \(\log (z_j^0)^\epsilon\) also go to  \(0\), so all that remains is \(-\zeta^0 \zeta^1/2\).
  Adding them up gives \cref{eq:crossing-volume-pinched}.
\end{proof}

\begin{lemma}
  \label{lemma:partition-function-flattening-dependence}
  The volume \(\volf(c, \mathfrak{f})\) at a crossing \(c\) depends on the remaining components of the flattening \(\mathfrak{f}\) via \(2\pi^2 i\) times
  \begin{gather*}
    \mu_1:
    -\epsilon(\beta_{1'} - \beta_{1}) + \gamma_W - \gamma_S
    \\
    \mu_2:
    \epsilon(\beta_{2'} - \beta_{2}) + \gamma_S - \gamma_E
    \\
    \beta_1:
    \gamma_W - \gamma_N - \epsilon \mu_1
    \\
    \beta_2:
    \gamma_S - \gamma_W + \epsilon \mu_2
    \\
    \beta_{1'}:
    \gamma_E - \gamma_S + \epsilon \mu_1
    \\
    \beta_{2'}:
    \gamma_N - \gamma_E - \epsilon \mu_2
    \\
    \gamma_{N}:
    \beta_1 - \beta_{2'}
    \\
    \gamma_{W}:
    \beta_2 - \beta_1 - \mu_1
    \\
    \gamma_{S}:
    \beta_{1'} + \mu_1 - \beta_2 - \mu_2
    \\
    \gamma_{E}:
    \beta_{2'} + \mu_2 - \beta_{1'}
  \end{gather*}
  For example, this means that if \(\mathfrak f'\) agrees with \(\mathfrak f\) except that \(\mu_1\) is replaced by \(\mu_1 + k\) for \(k \in \ZZ\), then 
  \[
    \volf(c, \mathfrak f') = 
    2\pi^{2} i
    \left[ \epsilon(\beta_1 - \beta_{1'}) + \gamma_W - \gamma_S \right]
    +
    \volf(c, \mathfrak{f}).\qedhere
  \]
\end{lemma}

\begin{proof}
  By \cref{thm:our-flattening}, when we change the value of \(\beta_1\) to \(\beta_1 + k\), we are changing \(\zeta_N^0\) to \(\zeta_N^0 - 2\pi i \epsilon k\) and the same for \(\zeta_W^0\).
  Because \(\tau_N\) has sign \(\epsilon\) and \(\tau_W\) has sign \(-\epsilon\) we can apply \cref{lemma:flattening-dependence} to see that \(\volf\) changes by
  \[
    \begin{aligned}
      &
      \pi k \left( -(\epsilon)(-\epsilon) \zeta_N^1 - (-\epsilon)(-\epsilon) \zeta_W^1 \right)
      \\
      &=
      2\pi^2 i k \left( -(\epsilon)(-\epsilon) (\kappa - \gamma_N) - (-\epsilon)(-\epsilon) (\kappa - \gamma_W + \epsilon \mu_1) \right)
      \\
      &=
      2\pi^2 i k \left( (\kappa - \gamma_N) -  (\kappa - \gamma_W + \epsilon \mu_1) \right)
      \\
      &=
      2\pi^2 i k \left( \gamma_W - \gamma_N -  \epsilon \mu_1) \right)
    \end{aligned}
  \]
  as claimed.
  Here we assumed the crossing was not pinched but our definition of \(\volf\) for pinched crossings was made so that it has the same dependence on the flattening in both cases.
  Repeating these computations for the remaining parameters finishes the proof.
\end{proof}

\begin{marginfigure}
  \centering
\begingroup%
  \makeatletter%
  \providecommand\color[2][]{%
    \errmessage{(Inkscape) Color is used for the text in Inkscape, but the package 'color.sty' is not loaded}%
    \renewcommand\color[2][]{}%
  }%
  \providecommand\transparent[1]{%
    \errmessage{(Inkscape) Transparency is used (non-zero) for the text in Inkscape, but the package 'transparent.sty' is not loaded}%
    \renewcommand\transparent[1]{}%
  }%
  \providecommand\rotatebox[2]{#2}%
  \newcommand*\fsize{\dimexpr\f@size pt\relax}%
  \newcommand*\lineheight[1]{\fontsize{\fsize}{#1\fsize}\selectfont}%
  \ifx\svgwidth\undefined%
    \setlength{\unitlength}{143.03116035bp}%
    \ifx\svgscale\undefined%
      \relax%
    \else%
      \setlength{\unitlength}{\unitlength * \real{\svgscale}}%
    \fi%
  \else%
    \setlength{\unitlength}{\svgwidth}%
  \fi%
  \global\let\svgwidth\undefined%
  \global\let\svgscale\undefined%
  \makeatother%
  \begin{picture}(1,0.98462123)%
    \lineheight{1}%
    \setlength\tabcolsep{0pt}%
    \put(0,0){\includegraphics[width=\unitlength,page=1]{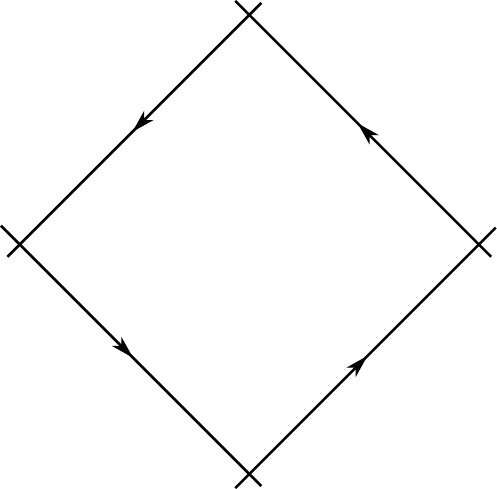}}%
    \put(0.25920426,0.63918978){\makebox(0,0)[lt]{\lineheight{1.25}\smash{\begin{tabular}[t]{l}$1$\end{tabular}}}}%
    \put(0.27312337,0.308078){\makebox(0,0)[lt]{\lineheight{1.25}\smash{\begin{tabular}[t]{l}$2$\end{tabular}}}}%
    \put(0.67648654,0.2950644){\makebox(0,0)[lt]{\lineheight{1.25}\smash{\begin{tabular}[t]{l}$3$\end{tabular}}}}%
    \put(0.68116965,0.6384653){\makebox(0,0)[lt]{\lineheight{1.25}\smash{\begin{tabular}[t]{l}$4$\end{tabular}}}}%
    \put(0.39291379,0.79803717){\makebox(0,0)[lt]{\lineheight{1.25}\smash{\begin{tabular}[t]{l}$\beta_1 - \beta_4$\end{tabular}}}}%
    \put(0.39234248,0.17298606){\makebox(0,0)[lt]{\lineheight{1.25}\smash{\begin{tabular}[t]{l}$\beta_3 - \beta_2$\end{tabular}}}}%
    \put(0.70487045,0.48114036){\makebox(0,0)[lt]{\lineheight{1.25}\smash{\begin{tabular}[t]{l}$\beta_4-\beta_3$\end{tabular}}}}%
    \put(0.09579986,0.48114399){\makebox(0,0)[lt]{\lineheight{1.25}\smash{\begin{tabular}[t]{l}$\beta_2 - \beta_1$\end{tabular}}}}%
  \end{picture}%
\endgroup%

  \caption{The dependence of each crossing on the choice of flattening for the internal region cancels out.}
  \label{fig:region-gluing-example-logarithm}
\end{marginfigure}

\begin{proof}[Proof of \cref{thm:diagram-flattening-dependence}]
  We already know that \(\volf\) does not depend on  the choice of the \(\kappa_j\).
  We next check that \(\volf\) does not depend on the region log-parameters \(\gamma_j\).
  This is an elementary combinatorial argument about oriented planar graphs.
  For example, applying \cref{lemma:partition-function-flattening-dependence} to the region in \cref{fig:region-gluing-example-logarithm} shows that the dependence on the log-parameter of the region is a multiple of
  \[
    \beta_1-\beta_4 + \beta_2 - \beta_1 + \beta_3-\beta_2 + \beta_4 - \beta_3 = 0.
  \]
  If we change the orientation of some or all of the strands we introduce \(\mu_j\) terms that again cancel.
  (Compare \cite[Figure 13]{McPhailSnyder2022}.)
  It is not hard to see this will work for any region of any link diagram.

  To prove \cref{eq:cvol-s-dependence} we need to total up the contributions of changing the parameters \(\set{\beta_j}\) and \(\set{\mu_j}\).
  It suffices to work one variable at a time.%
  \note{
    The term in \cref{eq:cvol-s-dependence} giving the dependence on the \(\set{\beta_j}\) and \(\set{\mu_j}\) does not look linear in them, but it is modulo \(4\pi^2 i \ZZ\).
  }
  Consider a segment \(j\) of the diagram with parameter \(\beta_j\)  in a flattening \(\mathfrak{f}\), and let \(\mathfrak{f}'\) be a flattening that is the same except it assigns this segment \(\beta_j + k\).
  We can use \cref{lemma:partition-function-flattening-dependence} to work out the change \(\Delta \volf = \volf(D, \mathfrak{f}') - \volf(D, \mathfrak{f})\) in the volume of the diagram caused by this.

  Our segment is oriented from an initial to a terminal crossing.
  Let \(\epsilon = 1\) if \(j\) is part of the overarc at its initial crossing and \(-1\) if it is part of the underarc, and similarly for \(\epsilon'\) and the terminal crossing.
  Some case by-case analysis shows that regardless of the signs \(\epsilon, \epsilon'\) the region parameters do not contribute to \(\Delta \volf\), while the \(\mu\) terms give
  \[
    \Delta \volf = 2\pi^2 i k \left[ \epsilon' \mu - \epsilon \mu\right]
  \]
  In turn, this says that
  \[
    \Delta \volf = 4 \pi^2 i k \eta
  \]
  where \(\eta\) is \(1\) if segment \(j\) is under-over and \(-1\) if it is over-under.
  The induced log-decorations have \(\mathfrak{s}'(\lon) = \mathfrak{s}(\lon) + k\), so
  \[
    \Delta \volf = 4 \pi i^2 \Delta \lambda \mu
  \]
  as required.

  For changes in \(\mu_j\) we need to add up terms across a component of \(D\).
  This turns out to be a logarithmic version of the computation in \cite[Section 4.3]{McPhailSnyder2022}, and the same arguments show that changing \(\mu\) gives 
  \[
    \Delta \volf = 
    -4\pi^2 i\Delta \mu (\mathfrak{s}'(\lonb) - \mathfrak{s}(\lonb))
  \]
  where \(\lonb\) is the \emph{blackboard-framed} longitude \(\lonb = w \mer + \lon\).
  Re-writing in terms of \(\lon\) gives the claimed dependence.
\end{proof}

\begin{marginfigure}
  \centering
\begingroup%
  \makeatletter%
  \providecommand\color[2][]{%
    \errmessage{(Inkscape) Color is used for the text in Inkscape, but the package 'color.sty' is not loaded}%
    \renewcommand\color[2][]{}%
  }%
  \providecommand\transparent[1]{%
    \errmessage{(Inkscape) Transparency is used (non-zero) for the text in Inkscape, but the package 'transparent.sty' is not loaded}%
    \renewcommand\transparent[1]{}%
  }%
  \providecommand\rotatebox[2]{#2}%
  \newcommand*\fsize{\dimexpr\f@size pt\relax}%
  \newcommand*\lineheight[1]{\fontsize{\fsize}{#1\fsize}\selectfont}%
  \ifx\svgwidth\undefined%
    \setlength{\unitlength}{106.44869614bp}%
    \ifx\svgscale\undefined%
      \relax%
    \else%
      \setlength{\unitlength}{\unitlength * \real{\svgscale}}%
    \fi%
  \else%
    \setlength{\unitlength}{\svgwidth}%
  \fi%
  \global\let\svgwidth\undefined%
  \global\let\svgscale\undefined%
  \makeatother%
  \begin{picture}(1,0.64456234)%
    \lineheight{1}%
    \setlength\tabcolsep{0pt}%
    \put(0,0){\includegraphics[width=\unitlength,page=1]{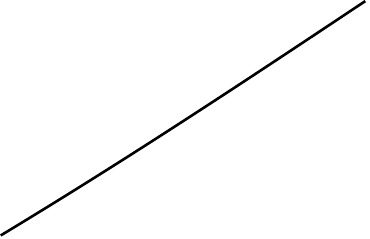}}%
    \put(0.45987414,0.07091752){\makebox(0,0)[lt]{\lineheight{1.25}\smash{\begin{tabular}[t]{l}$\gamma'$\end{tabular}}}}%
    \put(0.07940926,0.28933271){\makebox(0,0)[lt]{\lineheight{1.25}\smash{\begin{tabular}[t]{l}$\gamma$\end{tabular}}}}%
    \put(0.03008961,0.48661083){\color[rgb]{0,0,0}\makebox(0,0)[lt]{\lineheight{1.25}\smash{\begin{tabular}[t]{l}$\beta$\end{tabular}}}}%
    \put(0.73465445,0.01455239){\color[rgb]{0,0,0}\makebox(0,0)[lt]{\lineheight{1.25}\smash{\begin{tabular}[t]{l}$\beta'$\end{tabular}}}}%
    \put(0,0){\includegraphics[width=\unitlength,page=2]{longitude-contrib-logarithm.pdf}}%
  \end{picture}%
\endgroup%

  \caption{The dependence of \(\volf(c,\mathfrak{f})\) on the meridian log-parameter \(\mu\) of the {\color{accent} gold} strand is \(\pi i \Delta \mu( \epsilon(\beta' - \beta) + \gamma' - \gamma).\) }
  \label{fig:longitude-contrib-logarithm}
\end{marginfigure}

\subsection{\texorpdfstring{\(\volf\)}{𝒱} is a link invariant}
Now that we understand how \(\volf\) depends on the flattening we want to show that it is a well-defined link invariant.
Formally:
\begin{theorem}
  \label{thm:partition-function-is-link-invariant}
  Let \((L, \rho)\) be an oriented \(\slg\)-link with a choice \(\mathfrak{s}\) of log-decoration, and let \((D, \chi, \mathfrak{f})\) be a flattened diagram of \(D\) inducing \(\mathfrak{s}\) as in Definition \ref{def:induced-boundary-flattening}.
  Then
  \(
    \volf(L, \rho)(\mathfrak{s})
  \)
  does not depend on the choice of \((D, \chi, \mathfrak{f})\).
  Furthermore, if \(\rho'\) is conjugate to \(\rho\), 
  \(
    \volf(L, \rho')(\mathfrak{s})
    =
    \volf(L, \rho)(\mathfrak{s}).
  \)
\end{theorem}

The essential reason this works is an algebraic property of \(\dilr\).
Say \((D,\chi)\) and \((D',\chi')\) are two different diagrams of \(L\) with the same holonomy representations.
The induced octahedral decompositions are different triangulations of the same topological space \(S^3 \setminus L\), so they are connected by a series of 3-2 moves (also known as Pachner moves, Matveev--Piergallini moves, or bistellar flips), as in \cite[Figure 3]{Baseilhac2005}.
These moves extend to shaped and flattened triangulations \cite[Sections 2.1.3,2.1.4]{Baseilhac2005}.
Because the dilogarithm is also invariant under these moves (up to elements of \(2\pi^2 i\ZZ\)) the dilogarithm sum in
\cref{thm:volume-neumann} depends only on \(L\) and \(\rho\).

We want to extend this argument slightly to include \emph{pinched} diagrams where some of the tetrahedra become geometrically degenerate.
This could probably be done by extending the flattened 3-2 moves to allow (perhaps only some) degenerate tetrahedra.
We prefer to take a slightly different approach inspired by the biquandle calculus of \citeauthor{Blanchet2018} \cite{Blanchet2018}.
Just as we can prove a function on link diagrams is a link invariant by showing it is unchanged under Reidemeister moves, we can prove our function  \(\volf(D, \chi)(\mathfrak{s})\) is an invariant of links with extra structure by showing it is invariant under decorated Reidemeister moves.

\begin{lemma}[Invariance under the RI move]
  Let \((D, \chi)\) be a shaped link diagram.
  Apply an RI move to remove a kink and obtain another shaped diagram \((D', \chi')\).
  Then
  \[
    \volf(D, \chi)(\mathfrak{s}) = \volf(D', \chi')(\mathfrak{s})
  \]
  for any choice of log-decoration \(\mathfrak{s}\).
\end{lemma}

\begin{proof}
  Choose a flattening \(\mathfrak{f}\) of \((D,\chi)\) and let \(\mathfrak{s}\) be the induced log-decoration.
  Consider a kink in \(D\), as in \cref{fig:kink}.
  \begin{marginfigure}
    \centering
\begingroup%
  \makeatletter%
  \providecommand\color[2][]{%
    \errmessage{(Inkscape) Color is used for the text in Inkscape, but the package 'color.sty' is not loaded}%
    \renewcommand\color[2][]{}%
  }%
  \providecommand\transparent[1]{%
    \errmessage{(Inkscape) Transparency is used (non-zero) for the text in Inkscape, but the package 'transparent.sty' is not loaded}%
    \renewcommand\transparent[1]{}%
  }%
  \providecommand\rotatebox[2]{#2}%
  \newcommand*\fsize{\dimexpr\f@size pt\relax}%
  \newcommand*\lineheight[1]{\fontsize{\fsize}{#1\fsize}\selectfont}%
  \ifx\svgwidth\undefined%
    \setlength{\unitlength}{55.43417931bp}%
    \ifx\svgscale\undefined%
      \relax%
    \else%
      \setlength{\unitlength}{\unitlength * \real{\svgscale}}%
    \fi%
  \else%
    \setlength{\unitlength}{\svgwidth}%
  \fi%
  \global\let\svgwidth\undefined%
  \global\let\svgscale\undefined%
  \makeatother%
  \begin{picture}(1,1.46606328)%
    \lineheight{1}%
    \setlength\tabcolsep{0pt}%
    \put(0,0){\includegraphics[width=\unitlength,page=1]{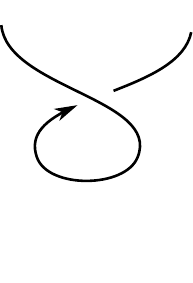}}%
    \put(0.28406465,0.29008601){\makebox(0,0)[lt]{\lineheight{1.25}\smash{\begin{tabular}[t]{l}$\beta'$\end{tabular}}}}%
    \put(0.39278587,0.67929231){\makebox(0,0)[lt]{\lineheight{1.25}\smash{\begin{tabular}[t]{l}$\gamma'$\end{tabular}}}}%
    \put(0.41017677,1.32625062){\makebox(0,0)[lt]{\lineheight{1.25}\smash{\begin{tabular}[t]{l}$\gamma$\end{tabular}}}}%
    \put(0.36899018,0.03875166){\makebox(0,0)[lt]{\lineheight{1.25}\smash{\begin{tabular}[t]{l}$\gamma''$\end{tabular}}}}%
    \put(1.04697098,1.11707732){\makebox(0,0)[lt]{\lineheight{1.25}\smash{\begin{tabular}[t]{l}$\beta''$\end{tabular}}}}%
    \put(-0.14871975,1.08963116){\makebox(0,0)[lt]{\lineheight{1.25}\smash{\begin{tabular}[t]{l}$\beta$\end{tabular}}}}%
  \end{picture}%
\endgroup%

    \caption{A kink in the diagram.}
    \label{fig:kink}
  \end{marginfigure}
  The shapes at a kink always satisfy
  \begin{align*}
    \chi_1 = \chi_{2'} = (a, b, m)
    \text{ and }
    \chi_2 = \chi_{1'} = (a', mb, m) 
  \end{align*}
  so in particular the shapes \(\chi_1 = \chi_{2'}\) of the incoming and outgoing strands are the same.%
  \note{%
    This corresponds to the fact that our braiding defines a biquandle and not a birack.
  }
  In terms of the flattening parameters, the volume of the kink is
  \[
    2\pi^2 i (\beta'' - \beta)(\gamma'' - \gamma)
    + 2\pi^2 i \mu ( \beta'' + \beta - 2 \beta' + 2\mu)
  \]
  However, in order to be able to remove the kink there are some restrictions on the flattening.
  Specifically, after applying the RI move the incoming and outgoing segment become the same segment, so we must have  \(\beta'' = \beta\), and the volume of the kink is
  \[
    4 \pi^2 i \mu ( \beta -  \beta' + \mu).
  \]
  This seems like a problem: we claimed that \(\volf\) was invariant under the RI move, but eliminating this crossing causes the volume to change!
  This is because eliminating the kink also changes the log-decoration from \(\mathfrak{s}\) to 
  \[
    \mathfrak{s'}(\lon) = \mathfrak{s}(\lon) - (\beta -\beta' + \mu)
  \]
  where \(\lon\) is the zero-framed longitude of the component containing the kink.
  We have computed that
  \[
    \volf(D, \chi)(\mathfrak{s}) - 4\pi^2 i \mu(\beta - \beta' + \mu) = \volf(D', \chi')(\mathfrak{s'}) 
  \]
  where \(\mathfrak{s}'\) is the log-decoration induced by the flattening \emph{after} the RI move.
  Comparing this to the dependence on the log-decoration given in \cref{eq:cvol-s-dependence} shows that
  \[
    \volf(D, \chi)(\mathfrak{s}')
    =
    \volf(D, \chi)(\mathfrak{s}) - 4\pi^{2} i \mu(\beta - \beta' + \mu)
    =
    \volf(D', \chi')(\mathfrak{s}').
  \]
  Because they agree for one log-decoration and obey the transformation rule \eqref{eq:cvol-s-dependence} they must agree for all choices of log-decoration.
  Our proof was for a particular type of kink (positive crossing, on the right-hand side) but the same argument works for the other cases.
\end{proof}

\begin{corollary}
  \label{thm:kink-value}
  If the flattening of the kink in \cref{fig:kink} is chosen so that \(\beta'' = \beta\), then its volume is
  \begin{equation}
    -4\pi^2 i \mu (\beta' - \beta - \mu).
  \end{equation}
\end{corollary}

\begin{lemma}[Invariance under RII moves]
  Let \((D, \chi)\) be a shaped link diagram.
  Apply an RII move to eliminate a pair of adjacent crossings of opposite sign and obtain another shaped diagram \((D', \chi')\).
  Then
  \[
    \volf(D, \chi)(\mathfrak{s}) = \volf(D', \chi')(\mathfrak{s})
  \]
  for any choice of log-decoration \(\mathfrak{s}\).
\end{lemma}
\begin{proof}
  Again we need to be slightly careful with our choice of flattening:
  the segment and region log-parameters at the input and output of our diagram must match in order to apply the RII move.
  It is a consequence of the definition of the braiding that the \emph{shapes} at the input and output are the same, so our requirement is just that we take the same logarithms of these numbers.
  When the flattening match it is easy to check directly that \(\volf\) vanishes on the diagram in Figure \cref{fig:RII-flattened.pdf_tex}.
  \begin{marginfigure}
\begingroup%
  \makeatletter%
  \providecommand\color[2][]{%
    \errmessage{(Inkscape) Color is used for the text in Inkscape, but the package 'color.sty' is not loaded}%
    \renewcommand\color[2][]{}%
  }%
  \providecommand\transparent[1]{%
    \errmessage{(Inkscape) Transparency is used (non-zero) for the text in Inkscape, but the package 'transparent.sty' is not loaded}%
    \renewcommand\transparent[1]{}%
  }%
  \providecommand\rotatebox[2]{#2}%
  \newcommand*\fsize{\dimexpr\f@size pt\relax}%
  \newcommand*\lineheight[1]{\fontsize{\fsize}{#1\fsize}\selectfont}%
  \ifx\svgwidth\undefined%
    \setlength{\unitlength}{127.50840569bp}%
    \ifx\svgscale\undefined%
      \relax%
    \else%
      \setlength{\unitlength}{\unitlength * \real{\svgscale}}%
    \fi%
  \else%
    \setlength{\unitlength}{\svgwidth}%
  \fi%
  \global\let\svgwidth\undefined%
  \global\let\svgscale\undefined%
  \makeatother%
  \begin{picture}(1,0.54077601)%
    \lineheight{1}%
    \setlength\tabcolsep{0pt}%
    \put(0,0){\includegraphics[width=\unitlength,page=1]{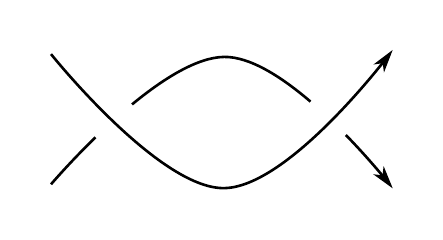}}%
    \put(0.02123498,0.42727488){\makebox(0,0)[lt]{\lineheight{1.25}\smash{\begin{tabular}[t]{l}$\beta_1$\end{tabular}}}}%
    \put(0.0222754,0.10385286){\makebox(0,0)[lt]{\lineheight{1.25}\smash{\begin{tabular}[t]{l}$\beta_2$\end{tabular}}}}%
    \put(0.89191869,0.42727242){\makebox(0,0)[lt]{\lineheight{1.25}\smash{\begin{tabular}[t]{l}$\beta_1$\end{tabular}}}}%
    \put(0.89295916,0.10385036){\makebox(0,0)[lt]{\lineheight{1.25}\smash{\begin{tabular}[t]{l}$\beta_2$\end{tabular}}}}%
    \put(0.0735819,0.27064916){\makebox(0,0)[lt]{\lineheight{1.25}\smash{\begin{tabular}[t]{l}$\gamma$\end{tabular}}}}%
    \put(0.84479915,0.27063899){\makebox(0,0)[lt]{\lineheight{1.25}\smash{\begin{tabular}[t]{l}$\gamma$\end{tabular}}}}%
  \end{picture}%
\endgroup%

    \caption{
      Showing that \(\volf\) vanishes on this diagram gives invariance  under the Reidemeister II move.
      The flattening should be compatible in the sense that the log-parameters on each boundary should match, as indicated.
    }
    \label{fig:RII-flattened.pdf_tex}
  \end{marginfigure}
\end{proof}

\begin{figure}
  \centering
  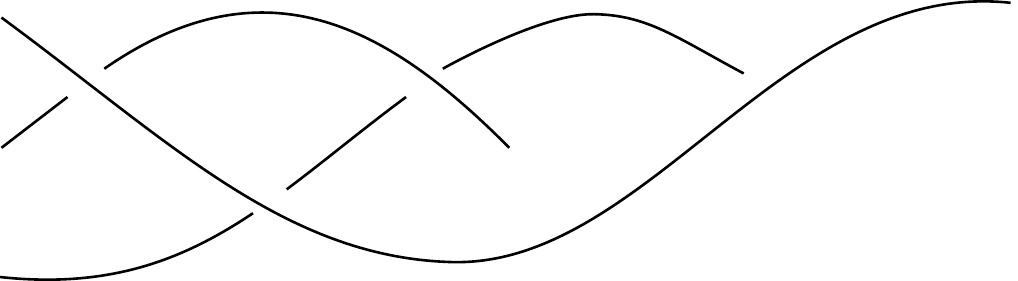
  \caption{Showing that \(\volf\) vanishes on this diagram for every choice of shapes \(\chi_i\) and appropriate flattenings gives invariance  under the Reidemeister III move.}
  \label{fig:RIII-colored}
\end{figure}

Once we know RI and RII hold, invariance under RIII moves follows from showing that the diagram in \cref{fig:RIII-colored} vanishes for every choice of shapes \(\chi_1, \chi_2, \chi_3\).
As before, we need some conditions on the flattenings:

\begin{lemma}
  \label{lemma:RIII-invariance}
  Let \(\chi\) be a shaping of the diagram \(D\) in \cref{fig:RIII-colored}, let \(\mathfrak{s}\) be a log-decoration sending the longitude of each component to \(0\), and let \(\mathfrak{f}\) be any flattening of \(D\) that inducing \(\mathfrak{s}\) that assigns the same log-parameter to the regions  \(R_{12}\), \(R_{12}'\) and to \(R_{23}\), \(R_{23}'\).
  Then
  \[
    \volf(D, \chi, \mathfrak{s})
    =
    0.
    \qedhere
  \]
\end{lemma}
\begin{proof}
  The octahedral decomposition gives an ideal triangulation of the complement of the braid in \cref{fig:RIII-colored}.
  When none of the crossings are pinched the tetrahedra are all nondegenerate and \(\volf\) is just the sum of their volumes computed using the lifted dilogarithm \(\dilr\).
  We claim that applying a series of 3-2 moves to these tetrahedra shows that  \(\volf \equiv 0 \pmod{2\pi^2 i \ZZ}\).
  This is a consequence of the usual, triangulation-based method \cite[Theorem 14.5]{Neumann2004} to compute complex volumes: if our claim were false then we could insert the diagram in \cref{fig:RIII-colored} into a link diagram and get a different volume for the same link.
  The condition on \(\mathfrak{s}\) ensures that we are not changing the log-decoration by inserting or removing the diagram.
 
  This proves the lemma when none of the crossings are pinched.
  When they are we can again rely on continuity and take the limit.
\end{proof}

\begin{proof}[Proof of \cref{thm:partition-function-is-link-invariant}]
  We know that \(\volf\) is invariant under the Reidemeister moves discussed above.
  Strictly speaking we should also consider variants of the Reidemeister moves with different orientations of the strands (not all left-to-right).
  This can be done either by repeating the proofs given above with different orientations (they go through the same) or by applying the results of the next subsection to reverse parts of link components as necessary.
  Either way we conclude that \(\volf(L, \rho, \mathfrak{s})\) does not depend on the choice of diagram \(D\) representing \(\rho\).
  
  The gauge-invariance claim also follows from invariance under Reidemeister moves.
  An important idea of \textcite[Section 3.3]{Blanchet2018} is that this gives gauge-invariance as well: we can always realize gauge transformations via adding new strands to our diagram, and then gauge-invariance follows from the Reidemeister moves, as in \cite[Theorem 5.10]{Blanchet2018}.
  (The extra category-theoretic requirements in that theorem are satisfied in our example for trivial reasons.)
\end{proof}

\subsection{Behavior under orientation changes}
So far all our definitions are for oriented links, so now we want to understand how they depend on this choice.
In addition, link exteriors in \(S^3\) inherit an orientation from \(S^3\), and we want to understand how the volume changes when taking the opposite orientation.

First, we show that the \emph{link} orientation doesn't matter: if we change the orientation of a diagram component and make appropriate changes to the shapes and flattenings we get the same volume.
\begin{definition}
  \label{def:inverse-shape}
  For any shape \(\chi = (a,b,m)\), the \defemph{inverse} shape is
  \begin{equation}
    \chi^{-1} \defeq (a^{-1}, bm, m^{-1}).
  \end{equation}
  It is straightforward to check that the holonomy \cite[Section 2.2]{McPhailSnyder2022} of a segment with shape \(\chi\) is equal to the holonomy of a segment with reversed orientation and shape \(\chi^{-1}\).
\end{definition}

\begin{definition}
  To reverse the orientation of a component of a tangle diagram, we
  \begin{itemize}
    \item change its orientation,
    \item replace all the shapes with their inverses,
    \item change log-meridian \(\mu\) to \(-\mu\), and
    \item replace each segment log-parameter \(\beta\) with \(\beta + \mu\).
  \end{itemize}
  Note that the region variables and region log-parameters are unaffected.
\end{definition}

\begin{proposition}
  The volume of a flattened tangle diagram does not change when reversing components.
\end{proposition}
\begin{proof}
  It suffices to check this crossing-by-crossing, which is elementary given \cref{def:crossing-volume}.
\end{proof}

On the other hand, changing the ambient orientation of our link should replace \(\volf\) with its complex conjugate.
This can also be represented diagrammatically:
\begin{definition}
  To take the \defemph{mirror image} of a flattened diagram, we
  \begin{itemize}
    \item reverse all the crossings,
    \item reverse the orientation of every component,
    \item replace all the shapes with their inverses,
    \item change each log-meridian \(\mu\) to \(-\mu\), 
    \item change each segment log-parameter \(\beta\) to \(\beta + \mu\), and
    \item change each region log-parameter \(\gamma\) to \(-\gamma\).
  \end{itemize}
\end{definition}
It is not obvious that this still gives a shaping of the new diagram, but once we have the right formula it is easy to check.

\begin{proposition}
  \label{thm:orientation-dep}
  Let \((\extr L, \rho, \mathfrak{s})\) be a link exterior in \(S^3\) represented by a flattened shaped diagram.
  Then
  \begin{enumerate}
    \item the mirror image of this diagram represents the same link exterior \((\overline{\extr{L}}, \rho, \mathfrak{s})\) but with the opposite orientation, and
    \item the volume changes by complex conjugation
  \[
    \volf(\extr L, \rho, \mathfrak{s}) = \overline{\volf(\extr L, \rho, \mathfrak{s})}.\qedhere
  \]
  \end{enumerate}
\end{proposition}
\begin{proof}
  The first claim is well-known: taking the mirror image of the diagram corresponds to applying an orientation-reversing homeomorphism to \(S^3\).
  The more complicated definition we give is so that the flattening on the mirror-image diagram corresponds to the same representation and log-decoration as on the original manifold.
  The second is another elementary check: we can confirm using \cref{def:crossing-volume} that the volume of the mirror image of a crossing is the conjugate of the original volume, and the general case follows.
\end{proof}

\section{Simple examples: solid tori and lens spaces}
Before giving the details of the general definition of the complex volume of a manifold we discuss a few simple examples.
The complex volume of a solid torus is important to know because these are the manifolds added during Dehn filling.
Once we understand solid tori it is simple to compute the complex volumes of lens spaces.
This computation also demonstrates how the dependence on the log-decorations drops out when Dehn filling, as discussed in general in the next section.

\subsection{Solid tori}
We now compute the complex volume of a solid torus as added during Dehn filling.
It turns out that its volume is \(2\pi\) times the complex length of the geodesic core of the torus, where the \defemph{complex length} of a geodesic is
\[
  \operatorname{len}_{\CC}(\mathfrak{g}) =
  \operatorname{length}(\mathfrak{g})
  +i
  \operatorname{torsion}(\mathfrak{g})
  \pmod{2\pi i \ZZ}.
\]
\begin{marginfigure}
  \centering
  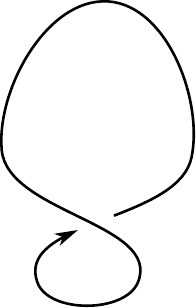
  \caption{A shaped diagram of the unknot that can be used to express \emph{any} log-decoration.}
  \label{fig:unknot-complement-diagram}
\end{marginfigure}

\begin{proposition}
  \label{thm:solid-torus-volume}
  Let \(\storus = D^2 \times S^1\) be a solid torus with meridian and longitude \(\mer , \lon \in \homol{\partial \storus; \ZZ}\); recall that \(\mer\)  is null-homologous in \(\storus\).
  Suppose that \(\rho : \pi_1(\storus) \to \slg\) is given by
  \[
    \rho(\lon)
    =
    \begin{pmatrix}
      a & -(a-m) b \\
      (a - 1/m)/b & m + m^{-1} -a
    \end{pmatrix}.
  \]
  Let \(\mathfrak{t}\) be a log-decoration for \((\storus, \rho)\).
  Then
  \begin{equation}
    \label{eq:solid-torus-volume}
    \volf(\storus, \rho)(\mathfrak{t})
    \equiv
    -4\pi^2 i\mathfrak{t}(\mer)\mathfrak{t}(\lon)
    \pmod{2\pi^2 i \ZZ}
    .
    \qedhere
  \end{equation}
\end{proposition}

\begin{proof}
  Think of \(\storus\) as being the exterior of the unknot \(\unknot\).
  Because \cref{fig:unknot-complement-diagram} is the closure of \cref{fig:kink} we can use the same shaping and flattening, as long as we choose \(\beta'' = \beta\).
  These choices define a log-decoration on the unknot complement with
  \begin{align*}
    \mathfrak{t}(\lon) &= \mu
    \\
    \mathfrak{t}(\mer) &= \beta' - \beta - \mu
  \end{align*}
  where \(\lon\) is the meridian of the knot and \(\mer\) is its longitude because \(\storus\) is \(\extr \unknot\) inside-out.
  Now we can apply \cref{thm:kink-value}.
\end{proof}

\begin{proposition}
  Suppose we are using \((\storus, \rho, \mathfrak{t})\) to fill in a boundary component of some other manifold \((M, \rho_0, \mathfrak{s})\).
   For a \((p,q)\) filling this implies that
   \[
     p \mathfrak{s}(\mer^M) + q \mathfrak{s}(\lon^M) = k
   \]
   for some \(k \in \ZZ\), where \(\mer^M\) and \(\lon^M\) are the meridian and longitude of the boundary component in question.
   Then when the log-decoration \(\mathfrak{t}\) on the filling torus is chosen to be compatible as in \cref{eq:boundary-log-dehn-filling}, which in this case means that
   \[
     p \mathfrak{s}(\mer^M) + q \mathfrak{s}(\lon^M) = k = \mathfrak{t}(\mer^W)
   \]
   the volume
   \[
     \cvol(\storus, \rho)(\mathfrak{t}) = 2 \pi k \operatorname{len}_\CC(\lon^\storus)
   \]
   is a multiple of the complex length of the core \(\lon^{\storus}\)of \(\storus\).
\end{proposition}
\begin{proof}
  We write \(\mer^M, \lon^M\) and \(\mer^\storus, \lon^{\storus}\) for the meridian and longitude of \(M\) and \(\storus\), respectively, and abbreviate
  \[
  \mathfrak{s}(\mer^M) = \mu \text{ and } \mathfrak{s}(\lon^M) = \lambda.
  \]
  In a \(p,q\) Dehn filling we are identifying \(\mer^\storus\) with \(p \mer^M + q \lon^M\).
  If we choose integers \(r,s\) with
  \[
    ps - qr =
    \begin{vmatrix}
      p & q \\
      r & s
    \end{vmatrix}
    =
    1
  \]
  then \((p\mer^M+ q \lon^M, r\mer^M+ s \lon^M)\) is a positively-oriented basis for the first homology of the boundary torus.
  We conclude that \(\lon^\storus\) will be identified with \(r \mer^\storus + s \lon^\storus\), which motivates the compatibility condition \eqref{eq:boundary-log-dehn-filling}.
  Since
  \[
     \mathfrak{t}(\mer^\storus) = k
  \]
  and
  \[
    \mathfrak{t}(\lon^{\storus}) = r \mu + s \lambda
  \]
  by \eqref{eq:solid-torus-volume} we have
  \[
    \cvol(\storus, \rho)(\mathfrak{t}) = -4\pi^2 i k(r \mu + s \lambda).
  \]
  Because \(- 2\pi i (r \mu + s \lambda) = \operatorname{len}_{\CC}(\lon^\storus)\) \cite[Lemma 4.2]{Neumann1985} this proves the claim.
  (Our \(\mu, \lambda\) differ from theirs by a factor of \(2\pi i\).)
\end{proof}

\subsection{Lens spaces}
We now compute the complex volumes of representations of the lens spaces \(\lens{-p,q}\); since \(\pi_1(\lens{-p,q}) = \ZZ/p\) there are not very many of these to consider.
While they are geometrically quite simple (in particular, \(\operatorname{Re} \cvol\) vanishes) they illustrate how our gluing results work in practice.

\begin{definition}
  Let \(p,q\) be coprime integers.
  The \(-p,q\) lens space \(\lens{-p,q}\) is the manifold obtained by \(p/q\) surgery on the unknot.
\end{definition}

\begin{proposition}
  Up to conjugacy every \(\rho : \pi_1(\lens{-p,q}) \to \slg\) is equivalent to
  \begin{equation}
    \rho_n(x) = 
    \begin{pmatrix}
      \omega_p^{n} & 0 \\
      0 & \omega_p^{-n}
    \end{pmatrix}
    \text{ for }
    0 \le n \le \lfloor p/2 \rfloor
  \end{equation}
  where \(\omega_p = \exp(2\pi i/p)\) and \(x\) is a generator of \(\pi_1(\lens{-p,q})\).
\end{proposition}
\begin{theorem}
  \begin{equation}
    \cvol(\lens{-p,q}, \rho_n) 
    \equiv
    4\pi^2 i \frac{n^2r}{p}.
    \pmod{2\pi^2i \ZZ}
  \end{equation}
  where \(r\) is any integer with \(qr \equiv -1 \pmod p\).
\end{theorem}
This agrees with \cite[Theorem 5.1]{Kirk1990}; we slightly extend their result to say that \[\volf(\lens{-p,q}, \rho_0) = 0,\] which makes sense because \(\rho_0\) is the trivial representation.
On the other hand, our method only computes the invariant modulo \(2\pi^2 i\), not \(4\pi^2i\), which is a general phenomenon when computing Chern-Simons invariants using the dilogarithm.

We compute the volume by thinking of \(\rho_n\) as a representation of the unknot exterior and then gluing in a solid torus.
Topologically this is just gluing two solid tori together, but thinking of it this way helps keep the meridians and longitudes straight.

Consider the diagram of the unknot in \cref{fig:unknot-complement-diagram}.
A shaping with holonomy conjugate to \(\rho\) is given by setting
\[
  \begin{gathered}
    \chi = (\omega_p^n, 1, \omega_p^n)
    \\
    \chi' = (\omega_p^{-n}, \omega_p^n, \omega_p^n)
  \end{gathered}
\]
To choose a flattening we need to choose logarithms \(\beta\) of \(1\), \(\beta'\) of \(\omega_p^n\), and \(\mu\) of \(\omega_p^n\).
(Strictly speaking we also need to choose logarithms \(\gamma\) associated to the regions of the diagram, but these do not affect the computation so we ignore them.)
We set
\begin{equation}
    \beta = (k-l),
    \quad
    \beta' = \frac{1}{p},
    \quad
    \mu = \frac{n}{p} + 2\pi i k
\end{equation}
for \(k,l \in \ZZ\).%
\note{
  We could set \(k = l = 0\), but by choosing them more generally we can explicitly see how gluing in a torus eliminates the log-decoration dependence.
}
The corresponding log-decoration \(\mathfrak{s}\) has
\begin{equation}
  \mathfrak{s}(\mer) = \mu = \frac{n}{p} + k,
  \quad
  \mathfrak{s}(\lon) = \lambda = \beta' - \beta - \mu = l
\end{equation}
where \(\lon\) is the zero-framed longitude.
Since there is only a single pinched crossing in our diagram, using \cref{thm:kink-value} we compute that
\[
  \volf(\unknot, \rho_n)
  = 4\pi^2 i \mu \lambda
  = 4\pi^2 i \left[ kl + \frac{nl}{p} \right]
  \equiv 4\pi^2 i\frac{nl}{p}.
  \pmod{2\pi^2 i \ZZ}
\]
Now let \( \storus\) be the solid torus we are attaching to \(\extr \unknot\).
Given \(r,s\) with \(ps - qr =1\), a compatible log-decoration \(\mathfrak{t}\) on \(\storus\) has
\[
  \begin{gathered}
    \mathfrak{t}(\mer^\storus) = p \mu + q \lambda = n + pk + ql
    \\
    \mathfrak{s}(\lon^{\storus}) = r\mu + s \lambda = nr/p + rk + sl
  \end{gathered}
\]
so by \cref{thm:solid-torus-volume} its volume is
\[
  \volf(\storus)
  = 
  4 \pi^2 i (n + pk + ql)(nr/p + rk + sl)
  \equiv
  4\pi^2 i \frac{n^2 r + nrql}{p}.
  \pmod{2\pi^2 i \ZZ}
\]
Because \(qr \equiv - 1 \pmod{p\ZZ}\) we see that
\[
  \volf(\storus)
  \equiv
  4\pi^2 i \frac{n^2 r - nl}{p}.
  \pmod{2\pi^2 i \ZZ}
\]
Adding this to \(\volf(\unknot, \rho_n)\) gives
\[
  \cvol(\lens{-p,q}, \rho_n)
  \equiv
  2\pi^2 i \frac{n^2 r}{p}
\]
as claimed.
In particular, we see that the dependence on the log-decoration \(\mathfrak{s}\) (that is, on \(l\)) has dropped out.

\section{The complex volume of a manifold}

In this section we show how to define the complex volume of a general manifold via surgery presentations and conclude the proof of our main results.

\subsection{Generalized surgery diagrams}
We first discuss our preferred representations of manifolds.
\begin{definition}
  A \defemph{generalized surgery presentation} is a link \(L_0\) in \(S^3\) whose components are labeled by one of
  \begin{enumerate}
    \item a rational number \(p/q\),
    \item the symbol \(\infty\), or
    \item nothing.
  \end{enumerate}
  Let \(L\) be the sublink of \(L_0\) consisting of components labeled \(\infty\) and \(L_{\partial}\) the sublink containing unlabeled components.
  Let \(M\) be the manifold \(M\) obtained by Dehn filling each rationally-labeled component according to its coefficient, then removing an open regular neighborhood of \(L_{\partial}\).
  \(M\) is a compact oriented manifold with one torus boundary component for each component of \(L_{\partial}\).,
  It contains an embedded link corresponding to \(L\), which by abuse of notation we also denote \(L\).
  We say that \(L_0\) \defemph{presents} the pair \((M, L)\).
  A representation \(\rho_0 : \pi_1(S^3 \setminus L_0) \to \slg\) is \defemph{compatible} with the presentation if for every component \(K\) of \(L_0\) labeled \(p/q\) we have
      \begin{equation}
        \label{eq:rho-dehn-condition}
        \rho(\mer)^p \rho(\lon)^q = 1
      \end{equation}
  where \(\mer\) and \(\lon\) are a meridian and longitude of \(K\).
  (Here, as usual, we assume \(p/q\) is in lowest terms.)
  In this case \(\rho_0\) induces a representation \(\rho : \pi_1(M \setminus L) \to \slg\) on the filled manifold and we say that \((L_0, \rho_0)\) \defemph{presents} the triple \((M, L, \rho).\)
  Typically we only consider \(\rho\) up to conjugacy (for example, we have not specified a basepoint).
\end{definition}

\begin{theorem}
  Every such \((M, L, \rho)\) is represented by some \((L_0, \rho_0)\).
\end{theorem}
\begin{proof}
  It is a standard result that any \((M,L)\) is represented by surgery along some \(L_0\).
  Adding the representations does not affect this.
\end{proof}

In practice, we need to represent these using shape coordinates, which we can do at the cost of conjugating the representation:

\begin{proposition}
  Let \((L_0, \rho_0)\) be a presentation of \((M, L, \rho)\), and let \(D\) be any diagram of \(L_0\).
  Then there is a representation \(\rho'\) conjugate to \(\rho\) and a presentation \((L_0, \rho_0')\) of \((M, L, \rho')\) for which there exists a shaping \(\chi\) of \(D\) with holonomy \(\rho_0'\).
\end{proposition}
\begin{proof}
  Apply part 2 of \cref{thm:shape-paper-results}.
\end{proof}

\subsection{Dehn filling}
We are now ready to given the general definition of \(\cvol\) and prove \cref{thm:cvol-for-manifolds-with-boundary,thm:connect-sum}.
We have already shown that \(\volf\) is an invariant of link exteriors, but it depends on the log-decorations chosen on the boundary components.
When we fill these in we need to make sure that the dependence on the filled components drops out.
We also need to check that \(\cvol\) does not depend on the chosen surgery presentation, which concludes the proof of \cref{thm:cvol-for-manifolds-with-boundary}.
Once this is done it is easy to prove \cref{thm:connect-sum}.

\begin{definition}
  For a link \(L\) in a \(3\)-manifold \(M\) we say a representation \(\rho : \pi_1(M \setminus L) \to \slg\) is \defemph{parabolic along \(L\)} if \(\tr \rho (\mer_j) = \pm 2\) for a meridian \(\mer_j\) of each component \(L_j\) of \(L\).
  This includes the case where \(\rho(\mer_j) = \pm 1\), in which case we say \defemph{trivial along \(L_j\)}.
\end{definition}

\begin{definition}
  \label{def:cvol}
  Let \(M\) be a compact oriented manifold, \(L\) a link in \(M\), and \(\rho : \pi_1(M \setminus L) \to \slg\) a representation.
  Let \((L_0, \rho_0)\) be a generalized surgery presentation of \((M, L, \rho)\) in which components \(L_1, \dots, L_h\) are \(p_k,q_k\) Dehn filled.
  Write \(\mer_k, \lon_k\) for the meridian and longitude of \(L_k\).
  Then for any log-decoration \(\mathfrak{s}\) the complex volume is given by
  \[
    \cvol(M, L, \rho)(\mathfrak{s})
    =
    \volf(L_0, \rho_0)(\mathfrak{s}_0)
    +
    \sum_{k=1}^{h} \volf(\storus, \rho_k)(\mathfrak{t}_k)
  \]
  where the representations are compatible with the Dehn filling as in \eqref{eq:rho-dehn-condition}:
  \[
    \rho_k(\mer_k)^{p_k} \rho_k(\lon_k)^{q_k} = 1
  \]
  and the log-decorations are compatible with the Dehn filling as in \eqref{eq:boundary-log-dehn-filling}:
  \begin{align*}
    p_k \mathfrak{s}_k(\mer_k)
    +
    q_k \mathfrak{s}_k(\lon_k)
    &=
    \mathfrak{t}_k(\mer_k^{\storus}).
  \end{align*}
\end{definition}

\begin{lemma}
  \label{lemma:log-decoration-indep-dehn}
  \(\cvol(M, L, \rho)\) does not depend on the log-decorations chosen on the Dehn-filled components of \(L_0\).
\end{lemma}
\begin{proof}
  For simplicity, focus on a single component.
  Write \(\mathfrak{s}_0\) for the log-decoration on \(\extr L\) we are filling.
  Any other log-decoration \(\mathfrak{s}_0'\) has
  \[
    \begin{aligned}
      \mathfrak{s}_0'(\mer) &= \mathfrak{s}_0(\mer) + 2\pi i k 
      \\
      \mathfrak{s}_0'(\lon) &= \mathfrak{s}_0(\lon) + 2\pi i l
    \end{aligned}
    \quad
    \text{ for }
    k,l \in \ZZ.
  \]
  Because of our condition 
  \begin{align*}
    \mathfrak{t}(\mer^{\storus}) &= p \mathfrak{s}_0(\mer) + q \mathfrak{s}_0(\lon)
  \end{align*}
  on the log-decoration on the filling torus \(\storus\) there are \(r,s \in \ZZ\) with \(ps - qr = 1\) and
  \begin{align*}
    \mathfrak{t}(\lon^{\storus}) &= r \mathfrak{s}_0(\mer) + s \mathfrak{s}_0(\lon).
  \end{align*}
  By elementary number theory, we can use these to rewrite \(k,l\) as
  \[
    \begin{aligned}
      \mathfrak{s}_0'(\mer) &= \mathfrak{s}_0(\mer) + 2\pi i (aq + bs)
      \\
      \mathfrak{s}_0'(\lon) &= \mathfrak{s}_0(\lon) - 2\pi i (ap + br) 
    \end{aligned}
    \quad
    \text{ for }
    a,b \in \ZZ.
  \]
  From \cref{thm:diagram-flattening-dependence} the change \(\volf(L_0, \rho_0)(\mathfrak{s}_0') - \volf(L_0, \rho_0)(\mathfrak{s}_0)\) is
  \begin{align*}
    \Delta \volf_{\extr{L_0}}
    &=
    4\pi^2 i \left[ (ap + br)\mathfrak{s}_0(\mer) + (aq + bs) \mathfrak{s}_0(\lon) \right].
    \\
    &=
    4\pi^2 i a \mathfrak{t}(\mer^{\storus})
    +
    4\pi^2 i b \mathfrak{t}(\lon^{\storus}).
  \end{align*}
  On the other hand, after changing \(\mathfrak{s}\) we have to change \(\mathfrak{t}\) to match \(\mathfrak{s}'\).
  The new values are
  \begin{align*}
    \mathfrak{t}'(\mer^{\storus})
    &= p \mathfrak{s}_0'(\mer) + q \mathfrak{s}_0'(\lon)
    \\
    &= p \mathfrak{s}_0(\mer) + q \mathfrak{s}_0(\lon)
    + b
    \\
    \mathfrak{t} '(\lon^{\storus})
    &= r \mathfrak{s}_0'(\mer) + s \mathfrak{s}_0'(\lon)
    \\
    &= p \mathfrak{s}_0(\mer) + q \mathfrak{s}_0(\lon)
    - a
  \end{align*}
  where we again used \(ps -qr =1\), so the change in the volume is
  \begin{align*}
    \Delta \volf_{W}
    =
    -4\pi^2 i a \mathfrak{t}(\mer^{\storus})
    -
    4 \pi^2 i  b \mathfrak{t}(\lon^{\storus})
  \end{align*}
  These cancel as claimed.
\end{proof}

\begin{lemma}
  \label{lemma:log-decoration-indep-parabolic}
  \(\cvol(M, L, \rho)\) does not depend on the log-decorations chosen on the boundary-parabolic components of \(L_0\) representing the cusps \(L\).
\end{lemma}
\begin{proof}
  On any such component \(\mathfrak{s}(\mer)\) and \(\mathfrak{s}(\lon)\)  are elements of \(\frac{1}{2} \ZZ\), so picking a different log-decoration changes \(\volf\) by 
  \[
    -4\pi^2 i \left[ \Delta \lambda \mathfrak{s}(\mer) - \Delta \mu \mathfrak{s}(\lon) \right]
    \in 2 \pi^2 i \ZZ.\qedhere
  \]
\end{proof}

\begin{lemma}
  \label{lemma:trivial-volume}
  Let \(L'\) be the link obtained from \(L\) by removing all the components along which \(\rho\) is trivial.
  Then
  \[
  \cvol(M, L, \rho) = \cvol(M, L', \rho)
  \]
  for every log-decoration.
\end{lemma}
\begin{proof}
  Let \(K\) be one of the removed components.
  Any shaping of any diagram of \(L\) assigns the shape \((\pm 1,1,\pm 1)\) to every segment of \(K\).
  It is straightforward to check that any crossing involving the shape \((\pm 1,1, \pm 1)\) is pinched and its volume is \(0 \pmod{2\pi^2 i}\).
\end{proof}

\begin{figure}
\begingroup%
  \makeatletter%
  \providecommand\color[2][]{%
    \errmessage{(Inkscape) Color is used for the text in Inkscape, but the package 'color.sty' is not loaded}%
    \renewcommand\color[2][]{}%
  }%
  \providecommand\transparent[1]{%
    \errmessage{(Inkscape) Transparency is used (non-zero) for the text in Inkscape, but the package 'transparent.sty' is not loaded}%
    \renewcommand\transparent[1]{}%
  }%
  \providecommand\rotatebox[2]{#2}%
  \newcommand*\fsize{\dimexpr\f@size pt\relax}%
  \newcommand*\lineheight[1]{\fontsize{\fsize}{#1\fsize}\selectfont}%
  \ifx\svgwidth\undefined%
    \setlength{\unitlength}{311.25002289bp}%
    \ifx\svgscale\undefined%
      \relax%
    \else%
      \setlength{\unitlength}{\unitlength * \real{\svgscale}}%
    \fi%
  \else%
    \setlength{\unitlength}{\svgwidth}%
  \fi%
  \global\let\svgwidth\undefined%
  \global\let\svgscale\undefined%
  \makeatother%
  \begin{picture}(1,0.29740708)%
    \lineheight{1}%
    \setlength\tabcolsep{0pt}%
    \put(0,0){\includegraphics[width=\unitlength,page=1]{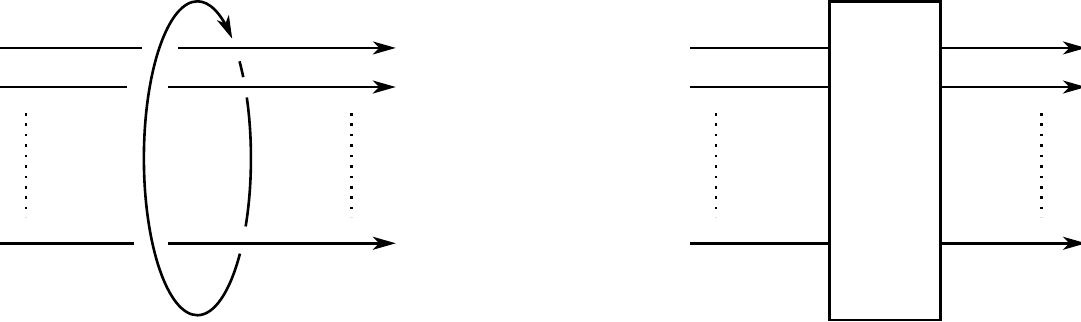}}%
    \put(0.79893684,0.14334394){\makebox(0,0)[lt]{\lineheight{1.25}\smash{\begin{tabular}[t]{l}$\mp 1$\end{tabular}}}}%
    \put(0.2441124,0.13479818){\makebox(0,0)[lt]{\lineheight{1.25}\smash{\begin{tabular}[t]{l}$\pm 1$\end{tabular}}}}%
  \end{picture}%
\endgroup%

  \caption{
    Showing that \(\cvol\) does not depend on the choice of surgery presentation follows from showing that these diagrams have equal volumes when the flattenings are chosen so that the through-strands have matching log-decorations.
    The box represents a full twist of the indicated sign.
  }
  \label{fig:blow-up}
\end{figure}

\begin{proof}[Proof of \cref{thm:cvol-for-manifolds-with-boundary}]
  Present \((M, L, \rho)\) as surgery on \((L_0, \rho_0)\).
  By \cref{thm:diagram-flattening-dependence} \(\volf(L_0, \rho_0)(\mathfrak{s}_0)\) is well-defined and does not depend on the conjugacy class of \(\rho_0\).
  The same theorem shows that it depends on the values of \(\mathfrak{s}_0\) on the unfilled components of \(L_0\) as claimed.
  By \cref{lemma:log-decoration-indep-dehn,lemma:log-decoration-indep-parabolic} it does not depend on the values of \(\mathfrak{s}_0\) on the other (filled and cusp) components of \(L_0\) and by \cref{lemma:trivial-volume} we can remove trivial-holonomy components.
  Finally, the orientation dependence follows from \cref{thm:orientation-dep}.

  It remains only to check that \(\cvol(M, L, \rho)\) does not depend on the choice of presentation  \((L_0, \rho_0)\).
  As usual, it suffices to check that it is invariant under the blowing up/down move on \(L_0\) shown in \cref{fig:blow-up}.
  When there are no strands this follows from \cref{lemma:trivial-volume}: for any representation \(\rho\) compatible with the framing a \(\pm 1\)-framed, unlinked unknot is always meridian-trivial, so it contributes no volume.

  For the general case we have to rely on 3-2 moves, as in the proof of \cref{lemma:RIII-invariance}.
  When none of the crossings in the diagram are pinched, invariance holds because the triangulations associated to each side are related by 3-2 moves.
  When they are pinched we can again rely on continuity and take the limit.
\end{proof}

One advantage of our method is that the volume is obviously additive under connect sums:

\begin{proof}[Proof of \cref{thm:connect-sum}]
  If \(L_0^1\) and \(L_0^2\) are presentations of \(M_1\) and \(M_2\), then their disjoint union \(L_0^1 \amalg L_0^2\) is a presentation of \(M_1 \# M_2\).
  It is obvious that \(\volf\) is additive under disjoint union (i.e.\@ under tensor product) of link diagrams.
\end{proof}

\subsection{General torus gluing}
We can now prove \cref{thm:gluing} in general: the special case where \(M_2\) is a disjoint union of solid tori is \cref{def:cvol,lemma:log-decoration-indep-dehn}.

We first consider the case where we are identifying a single boundary component \(T\) of \(M\) with a boundary component \(T'\) of \(M_2\).
Let \(\mer, \lon\) and \(\mer', \lon'\) be meridian-longitude pairs for the components in question.
The homomorphism identifying them can always be described as an integral matrix
\begin{equation}
  \label{eq:torus-gluing-matrix}
  \begin{bmatrix}
    \mer' \\
    \lon'
  \end{bmatrix}
  =
  \begin{bmatrix}
    r &  s \\
    p &  q
  \end{bmatrix}
  \begin{bmatrix}
    \mer \\
    \lon
  \end{bmatrix}
\end{equation}
with determinant \(-1 = rq - ps\), which corresponds to the fact that the identifying homomorphism is orientation-reversing.
We want to represent this matrix in terms of continued fractions.

\begin{definition}
  For integers \(a_1, \dots, a_k\), consider the continued fraction
  \begin{equation}
    \label{eq:cf-def}
    \cf{a_1, \dots, a_k}
    \defeq
    a_1 - 
    \cfrac{1}{a_2- \cfrac{1}{ \ddots - \cfrac{1}{a_k}}}
  \end{equation}
  and the associated sequences given by
  \begin{equation}
    \label{eq:cf-assoc}
    \begin{aligned}
      p_0 &= 0, & p_1 &= -1, & p_{i+1} &= a_i p_i  - p_{i-1}
      \\
      q_0 &= 1, & q_1 &= 0, & q_{i+1} &= a_i q_i  - q_{i-1}
    \end{aligned}
  \end{equation}
\end{definition}

\begin{lemma}
  \begin{equation}
    \label{eq:cf-ratio}
    \frac{p_{k+1}}{q_{k+1}} = \cf{a_1, \dots, a_k}
  \end{equation}
  and
  \begin{equation}
    \label{eq:cf-det}
    \det
    \begin{bmatrix}
      p_{k} & q_k \\
      p_{k+1} & q_{k+1} 
    \end{bmatrix}
    = -1
  \end{equation}
  for \(k \ge 1\).
\end{lemma}
\begin{proof}
  Use induction.
  For the first claim the identity \(\cf{a_1, \dots, a_{k-1}, a_k} = \cf{a_1, \dots, a_{k-1} - 1/a_k}\) is useful.
\end{proof}

Using the lemma, we can choose a continued fraction with
\begin{align*}
  \frac{p}{q}
  &=
  \frac{p_{k+1}}{q_{k+1}}
  =
  \cf{a_1, \dots, a_k}
  \\
  \frac{r}{s}
  &=
  \frac{p_{k}}{q_{k}}
  =
  \cf{a_1, \dots, a_{k-1}}
\end{align*}
Represent \((M_1, \rho_1)\) and \((M_2, \rho_2)\) with generalized surgery presentations \(L_0\) and \(L_0'\), respectively.
By abuse of notation we refer to the component of \(D\) representing \(T\) as \(T\), and similarly for \(D'\) and \(T'\).
We can pull out parts of \(T\) and \(T'\) to get a diagram of the form shown in \cref{fig:torus-gluing-general-diagram}.
For simplicity we have chosen \(L_0\) and \(L_0'\) to be disjoint, but this is not really necessary.

\begin{proposition}
  The link diagram in \cref{fig:torus-gluing-general-diagram} represents the manifold obtained by gluing \(T\) to \(T'\).
\end{proposition}
\begin{figure}
  \centering
  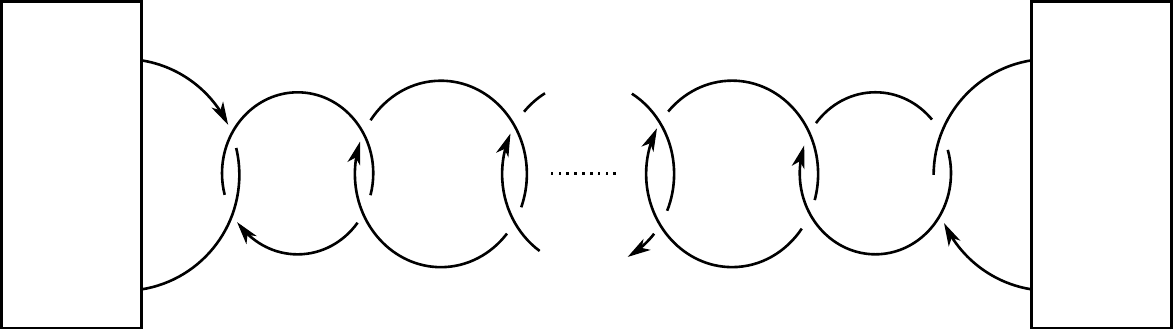
  \caption{
      Here \(L_0\) and \(L_0'\) are links presenting the manifolds \(M\) and \(M'\) and we have pulled out components representing the boundary components \(T\) and \(T'\) we are gluing.
      If we pick a continued fraction \(\cf{a_1, \dots, a_k}\) representing the homomorphism identifying \(T\) and \(T'\) then the link above represents the glued manifold.
    }
  \label{fig:torus-gluing-general-diagram}
\end{figure}
\begin{proof}
  This is a standard construction, as in \cite[Section 5.3]{Akbulut2016}.
  We call the \(0\)-framed unknot on the left \(U_0\), the  \(i\)th \(-a_i\)-framed
  \note{
    The right sign here is \(-a_i\), not \(a_i\).
    This has to do with the usual issue of keeping \(p/q\) and \(-p/q\) lens spaces straight.
  }
  unknot \(U_i\), and the \(0\)-framed unknot on the right \(U_{k+1}\), and write \((\mer_i, \lon_i)\) for their meridian-longitude pairs.
  Because the diagram in \cref{fig:torus-gluing-general-diagram} is relatively simple, it is not hard to work out the identifications
  \[
    \mer_1 = \mer \text{ and } \mer_k = \mer',
  \]
  and more generally
  \[
    \lon_i = \mer_{i+1} + \mer_{i-1}.
  \]
  (Alternately, these can be deduced from trying to compute the shaping in \cref{lemma:torus-gluing-general-shaping}.)

  Dehn filling along the \(0\)-framed components gives the relations
  \[
    \mer_0 = \lon \text{ and } \mer_{k+1} = \lon'
  \]
  When we Dehn fill according to the framings on the \(U_1, \dots, U_k\) we impose the relation
  \[
    -a_i \mer_i + \lon_i = 0, 1 \le i \le k
  \]
  which gives the recursive relation
  \begin{equation}
    \label{eq:meridian-recursion}
    \mer_{i+1} = a_i \mer_i - \mer_{i-1}
  \end{equation}
  in the glued manifold.
  We can now use the recurrence relation \cref{eq:cf-assoc} to conclude that
  \begin{align*}
    \mer_{i} = p_i \mer_i + q_{i} \mer_{i-1}
  \end{align*}
  for each \(i\).
  In particular, given the initial conditions
  \(
    \mer_0 = \lon, \mer_1 = \mer
  \)
  we get
  \begin{equation*}
    \lon'
    =
    \mer_{k+1}
    =
    p_{k+1} \mer_1 + q_{k+1} \mer_0
    =
    p_{k+1} \mer + q_{k+1} \lon
  \end{equation*}
  and similarly
  \begin{equation}
    \mer'
    =
    \mer_{k}
    =
    p_{k} \mer + q_{k} \lon.
  \end{equation}
  so we have identified the boundaries of \(T\) and \(T'\) according to \cref{eq:torus-gluing-matrix}.
\end{proof}

Following the notation in the above proof, we write \(m = \delta(\mer), \ell = \delta(\lon)\) for the meridian and longitude eigenvalues of \(T\), \(m', \ell'\) for \(T'\), and \(m_i\) for the meridian eigenvalues of the new components \(U_0, \dots, U_{k+1}\).
Our earlier identifications give \(m_0 = \ell, m_1 = m, m _{k} = m', m_{k+1} = \ell'\).

We now need to choose shapings on our combined diagram.
The simplest thing to do is choose them so that all the new crossings are pinched.
By conjugating \(\rho_2\) we can choose a shaping in which the \(b\)-variables of the exposed arcs both have the same value \(x\), and then we can assign \(b\)-variables as below:

\begin{figure*}
  \centering
  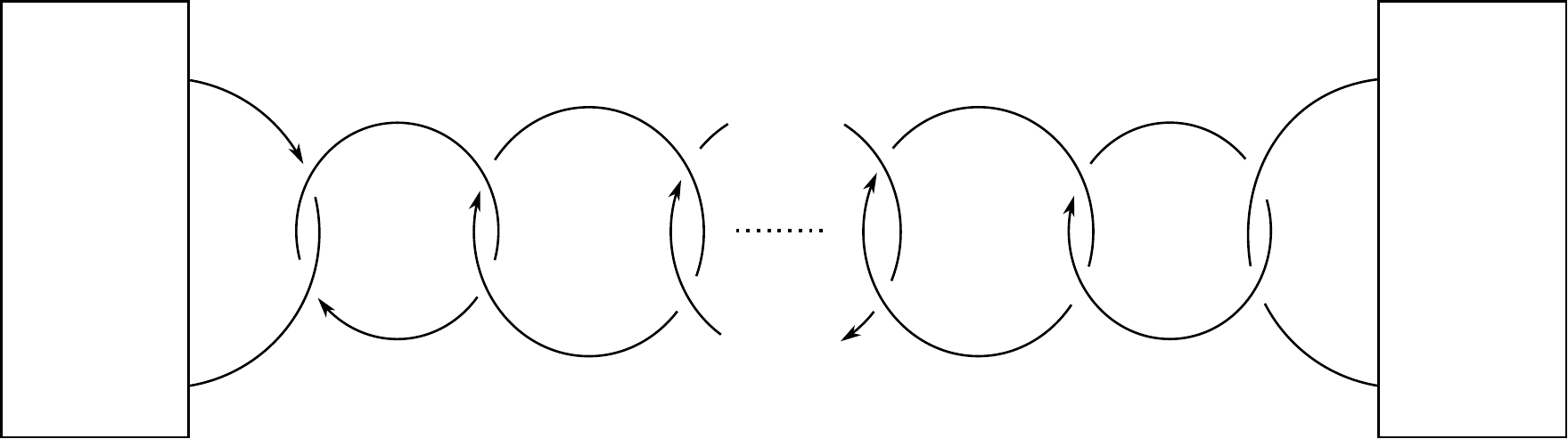
  \caption{
    Choosing the \(b\)-variables as indicated gives a shaping in which every crossing is pinched.
    We might have conjugate the holonomy representation of \(L_0'\) to achieve this, but that does not affect the value of \(\volf\).
  }
  \label{fig:torus-gluing-general-shaping}
\end{figure*}

\begin{lemma}
  \label{lemma:torus-gluing-general-shaping}
  There is a shaping of our diagram with the \(b\)-variables given in \cref{fig:torus-gluing-general-shaping}.
  In particular, we can choose all the new crossings to be pinched.
\end{lemma}
\begin{proof}
  The \(b\)-variables have been chosen so that all the crossings shown in the figure are pinched and the shapes adjacent to the crossings inside the boxes are the same as before we added the new components \(U_i\).
  This means that all we need to do is choose \(a\)-variables for the new diagram segments satisfying the gluing equations of the crossings show in \cref{fig:torus-gluing-general-shaping}.
  Because these crossings are pinched, this is easy to do.
\end{proof}

\begin{proof}[Proof of \cref{thm:gluing}]
  The point of doing all this is that \cref{fig:torus-gluing-general-shaping} is a shaped diagram presenting the glued manifold whose volume is easy to compute.
  There is an obvious flattening to pick: we choose some \(\beta\) with \(e^{2\pi i \beta} = x\) and pick the \(\mu_i\) so that
  \[
    \mu_{i+1} = a_i \mu_i - \mu_{i-1}.
  \]
  We then assign the horizontally-oriented strands \cref{fig:torus-gluing-general-shaping} the segment log-parameter \(\beta\) and the vertically-oriented strands segment log-parameter \(\beta + \mu_i\).
  This ensures that the flattening is compatible with the surgery coefficients on each component.

  Then, by definition the volume is
  \[
    \volf(L, \rho)(\mathfrak{s})
    +
    \volf(L', \rho')(\mathfrak{s}')
    +
    \text{ the volume of the new crossings,}
    +
    \text{ volumes of the added solid tori.}
  \]
  Because \(\volf(L, \rho)(\mathfrak{s}) = \cvol(M_1, L_1, \rho_1)(\mathfrak{s}_1)\) and \(\volf(L', \rho', \mathfrak{s}') = \cvol(M_2, L_2, \rho_2)\) we just need to show that the last two terms cancel.
  Because the new crossings are all pinched this is easy to check.
  Finally, we can show that the dependence on the log-decorations of \(T\) and \(T'\) drops out by the same computation as in the proof of \cref{lemma:log-decoration-indep-dehn}.

  This proves the gluing formula in \cref{thm:gluing} when we are identifying a single boundary component.
  To identify more than one component, we can repeat this process; we did not actually use the fact that the diagrams of \(L\) and \(L'\) are separated, even though this is depicted in \cref{fig:torus-gluing-general-diagram}.
\end{proof}

\printbibliography

\end{document}